\documentclass[reqno,11pt]{amsart}
\usepackage{amsfonts}
\usepackage{a4wide}
\usepackage{color}
\usepackage{mathrsfs}
\usepackage{mathtools}
\usepackage{amsmath}
\usepackage{amssymb}
\usepackage{bbm}
\usepackage{esint}
\usepackage{nicefrac}
\numberwithin{equation}{section}
\usepackage[colorlinks,citecolor=green,linkcolor=red]{hyperref}

\usepackage[latin1]{inputenc}
\input xy
\xyoption{all}

\newtheorem{theorem}{Theorem}[section]
\newtheorem{Ca}[theorem]{Corollary}
\newtheorem{Th}[theorem]{Theorem}
\newtheorem{Lm}[theorem]{Lemma}
\newtheorem{Prop}[theorem]{Proposition}
\newtheorem{Def}[theorem]{Definition}

\newtheorem{Remark}[theorem]{Remark}

\title{Traces of Sobolev spaces to piecewise Ahlfors--David regular sets}
\author{Alexander I. Tyulenev}
\address{Steklov Mathematical Institute of Russian academy of Sciences}
\email{tyulenev-math@yandex.ru,tyulenev@mi-ras.ru}
\begin{document}
\date{\today}
\allowdisplaybreaks
\begin{abstract}
Let $(\operatorname{X},\operatorname{d},\mu)$ be a metric measure space with uniformly locally doubling measure $\mu$. Given $p \in (1,\infty)$, assume that $(\operatorname{X},\operatorname{d},\mu)$
supports a weak local $(1,p)$-Poincar\'e inequality. We characterize trace spaces of the first-order Sobolev $W^{1}_{p}(\operatorname{X})$-spaces  to subsets $S$ of $\operatorname{X}$ that can be represented as a finite union $\cup_{i=1}^{N}S^{i}$, $N \in \mathbb{N}$, of Ahlfors--David regular subsets $S^{i} \subset \operatorname{X}$, $i \in \{1,...,N\}$ of different codimensions. Furthermore, we explicitly compute the corresponding trace norms
up to some universal constants.
\end{abstract}
\maketitle
\tableofcontents
\section*{Introduction}
Let $p \in (1,\infty)$ and let $(\operatorname{X},\operatorname{d},\mu)$ be a metric measure space supporting
a weak local $(1,p)$-Poincar\'e inequality (see Section 2 for details). The problem of traces of Sobolev $W_{p}^{1}(\operatorname{X})$-spaces to different closed subsets $S \subset \operatorname{X}$
has attracted a lot of attention in the recent years \cite{GBIZ,SakSot,Shv1,TV1,TV2,T1} (see also the references therein).
Furthermore, a closely related problem of traces of the first-order Sobolev $W_{p}^{1}(\Omega)$-spaces (initially defined on a domain $\Omega \subset \operatorname{X}$)
to the boundary $\partial \Omega$ is also of great interest \cite{GKS,GS,Maly}.

However, in all above papers it was assumed that $S$ satisfies a some sort of codimensional Ahlfors--David-type regularity condition, i.e., $S \in \mathcal{ADR}_{\theta}(\operatorname{X})$ for some $\theta \geq 0$ (see Definition \ref{Def.Ahlfors_David}).
Unfortunately, given $\theta \geq 0$, the class $\mathcal{ADR}_{\theta}(\operatorname{X})$ is too narrow. For example, one can construct
simple planar rectifiable curves that do not belong to $\mathcal{ADR}_{\theta}(\mathbb{R}^{2})$ for any $\theta \geq 0$ \cite{T2}.

In \cite{T5}, given $\theta \geq 0$, the class of all lower codimension-$\theta$ regular sets $\mathcal{LCR}_{\theta}(\operatorname{X})$ was introduced (see Definition \ref{Def.lower_content}).
Given $\theta \geq 0$, we have $\mathcal{ADR}_{\theta}(\operatorname{X}) \subset \mathcal{LCR}_{\theta}(\operatorname{X})$ \cite{T5}, but typically the class $\mathcal{LCR}_{\theta}(\operatorname{X})$
is much broader than $\mathcal{ADR}_{\theta}(\operatorname{X})$ \cite{TV1,T5}.
Given $\theta \in [0,p)$ and a closed set $S \in \mathcal{LCR}_{\theta}(\operatorname{X})$,
the author obtained in \cite{T5} sharp intrinsic descriptions of traces of $W_{p}^{1}(\operatorname{X})$-spaces to $S$. The results of \cite{T5} cover all previously known results concerning traces of $W_{p}^{1}(\operatorname{X})$-spaces to different subsets $S \subset \operatorname{X}$.
At the same time, due to the high generality, the corresponding criteria
given in \cite{T5} are quite abstract. Indeed, they were based on some special sequences of measures called $\theta$-regular.
The corresponding constructions of those measures given in \cite{T5} were based on some nontrivial techniques including Chryst's dyadic cubes and
the locally $\ast$-weak convergence of measures. This fact makes the corresponding criteria quite difficult
for applications.

In fact, finding explicit simple constructions of $\theta$-regular sequences of measures is a subtle problem.
It is natural to find some particular cases of sets $S \in \mathcal{LCR}_{\theta}(\operatorname{X}) \setminus \mathcal{ADR}_{\theta}(\operatorname{X})$, $\theta \geq 0$ for which
one can easily built the corresponding $\theta$-regular sequences of measures. A natural step in this direction is to consider piecewise
Ahlfors--David regular sets $S$, i.e., sets $S$ that can be represented as a
a finite number of pieces of Ahlfors--David regular sets of different codimensions. More precisely, given $p \in (1,\infty)$, we assume that $S=\cup_{i=1}^{N}S^{i}$, for some $N \in \mathbb{N}$, $N \geq 2$, where for each $i \in \{1,...,N\}$, $S^{i} \in \mathcal{ADR}_{\theta_{i}}(\operatorname{X})$ and $0 \le \theta_{1} < \theta_{2} < ... < \theta_{N} < p$. Given $\theta \in [\theta_{N},p)$, we construct explicit examples
of $\theta$-regular sequences of measures concentrated on $S$ and obtain explicit sharp intrinsic descriptions of traces of Sobolev $W_{p}^{1}(\operatorname{X})$-spaces
to $S$.

As far as we know, the results of the present paper are new and can not be obtained via the previously known methods.
We strongly believe that our results are quite transparent and can be effectively used in boundary value problems for partial differential equations.

In conclusion, we should mention that while the methods of \cite{T5} are capable of dealing with sets $S$ composed of
infinitely many  Ahlfros--David regular pieces of different codimensions, it is difficult to make the corresponding criteria transparent. In the
present paper, we essentially use the fact that $N < +\infty$. Furthermore, the corresponding intermediate constants in the present paper depend essentially on $N$.
Note, however, that in \cite{T2} the author made the first attempt to obtain explicit examples of $1$-regular sequences of measures together with transparent
trace criteria for the case of planar rectifiable curves of positive lengths and without self-intersections.
Clearly, such curves can not be obtained as a finite union of Ahflors--David regular sets in general.

We split the main results of the present paper into \textit{two parts}. The \textit{first part} corresponds to the case when $\theta_{i} > 0$ for all $i \in \{1,...,N\}$.
This case is technically simpler because the set $S$ is porous. In the \textit{second part} we assume that $\theta_{1}=0$ and $\theta_{i} > 0$ for all $i \in \{2,...,N\}$.
This case is more complicated because $S$ is not necessary porous.

\section{Necessary background and auxiliary results}

\subsection{The assumptions}
Throughout the paper we fix a \textit{metric measure space} $\operatorname{X}=(\operatorname{X},\operatorname{d},\mu)$, where $(\operatorname{X},\operatorname{d})$ is a \textit{complete separable}
metric space and $\mu$ is a \textit{Borel regular locally finite outer measure} on $\operatorname{X}$ with $\operatorname{supp}\mu=\operatorname{X}$
satisfying the \textit{uniformly locally doubling condition}, i.e., for each $R > 0$,
\begin{equation}
\label{eqq.doubling}
C_{\mu}(R):=\sup\limits_{r (0,R]}\sup\limits_{x \in \operatorname{X}}\frac{\mu(B_{2r}(x))}{\mu(B_{r}(x))} < +\infty,
\end{equation}
where $B_{r}(x)$ is the closed ball of radius $r$ centered at $x$, i.e.,
\begin{equation}
\label{eqq.closed_balls}
B_{r}(x):=\{y \in \operatorname{X}:\operatorname{d}(x,y) \le r\}.
\end{equation}
Furthermore, by a ball we always mean a \textit{closed ball} $B=B_{r}(x)$ for some $r \geq 0$ and $x \in \operatorname{X}$.
Clearly, if one consider a given ball $B$ just as a subset of $\operatorname{X}$, then it can happen that its center and its radius are not uniquely
determined. Hence, in the sequel we always consider a given ball $B$ together with some fixed center $x_{B}$ and fixed radius $r_{B}$.
Given a ball $B=B_{r}(x)$ and a constant $c \geq 0$, we set $cB:=B_{cr}(x)$.

We say that a family $\mathcal{F}$ of subsets of $\operatorname{X}$ is \textit{disjoint} if $F_{1} \cap F_{2} = \emptyset$
for different sets $F_{1},F_{2} \in \mathcal{F}$.

By $\operatorname{LIP}(\operatorname{X})$ we denote the linear space of all real valued Lipschitz functions on $\operatorname{X}$, i.e., $f \in \operatorname{LIP}(\operatorname{X})$
if and only if there is $L_{f} \geq 0$ such that
$$
|f(x)-f(y)| \le L_{f}\operatorname{d}(x,y), \quad (x,y) \in \operatorname{X} \times \operatorname{X}.
$$

By a \textit{measure on} $\operatorname{X}$ we always mean a nonzero Borel regular (outer) measure $\mathfrak{m}$ with $\operatorname{supp}\mathfrak{m} \subset \operatorname{X}$.
We say that $\mathfrak{m}$ is locally finite if $\mathfrak{m}(B_{r}(x)) < +\infty$ for all $x \in \operatorname{X}$ and all $r \in [0,+\infty)$.
Given $p \in [1,\infty)$, by $L_{p}(\mathfrak{m})$ ($L^{loc}_{p}(\mathfrak{m})$) we denote the linear space of all $\mathfrak{m}$-equivalence classes of $p$-integrable (locally $p$-integrable)
functions $f:\operatorname{supp}\mathfrak{m} \to [-\infty,+\infty]$.
Given a Borel regular locally finite (outer) measure $\mathfrak{m}$ on $\operatorname{X}$, for each $f \in L_{1}^{loc}(\mathfrak{m})$, and every Borel set $G \subset \operatorname{X}$
with $\mathfrak{m}(G) < +\infty$, we put
\begin{equation}
f_{G,\mathfrak{m}}:=\fint\limits_{G}f(x)\,d\mathfrak{m}(x):=
\begin{cases}
\frac{1}{\mathfrak{m}(G)}\int\limits_{G}f(x)\,d\mathfrak{m}(x), \quad \mathfrak{m}(G) > 0;\\
0, \quad \mathfrak{m}(G)=0.
\end{cases}
\end{equation}
We also set
\begin{equation}
\label{eqq.different_overagings}
\mathcal{E}_{\mathfrak{m}}(f,G):=\inf_{c \in \mathbb{R}}\fint_{G}|f(x)-c|\,d\mathfrak{m}(x), \quad \mathcal{OSC}_{\mathfrak{m}}(f,G):=\fint_{G}\fint_{G}|f(x)-f(y)|\,d\mathfrak{m}(x)d\mathfrak{m}(y).
\end{equation}
One can easily verify that (see Section 2 in \cite{T5} for the proof)
\begin{equation}
\label{eqq.averaging_comparison}
\mathcal{E}_{\mathfrak{m}}(f,G) \le \mathcal{OSC}_{\mathfrak{m}}(f,G) \le 2 \mathcal{E}_{\mathfrak{m}}(f,G).
\end{equation}

Finally, throughout the paper we \textit{fix a parameter} $p \in (1,\infty)$ and assume that the space $\operatorname{X}$ supports a \textit{weak local $(1,p)$-Poincar\'e inequality}, i.e.,
for any $R > 0$, there are constants $C=C(R) > 0$, $\lambda=\lambda(R) \geq 1$ such that, for each  $f \in \operatorname{LIP}(\operatorname{X})$,
\begin{equation}
\label{eq.Poincare}
\mathcal{E}_{\mu}(f,B_{r}(x)) \le C r \Bigl(\fint\limits_{B_{\lambda r}(x)}(\operatorname{lip}f(y))^{p}\,d\mu(y)\Bigr)^{\frac{1}{p}} \qquad \text{for all} \quad (x,r) \in \operatorname{X} \times (0,R],
\end{equation}
where $\operatorname{lip}f(y):=\varlimsup_{z \to y} |f(y)-f(z)|/\operatorname{d}(y,z)$ provided that $y \in \operatorname{X}$ is an accumulation point and $\operatorname{lip}f(y)=0$ provided that $y$ is an isolated point.

Our assumptions on the space $\operatorname{X}$ are quite typical in the modern Geometric Analysis and imply some nice properties of $\operatorname{X}$.
The reader can consult the beautiful monograph \cite{HKST} for the detailed exposition of metric measure 
spaces satisfying assumptions adopted in this paper.
We have the following result (see Proposition 2.23 in \cite{T5}).
\begin{Prop}
\label{Prop.volume_decay}
For each $R > 0$, there is $Q=Q(R) > 0$ such that the measure $\mu$ has the relative volume decay property of order $Q$, i.e., there exists $C(R,Q) > 0$ such that
\begin{equation}
\label{eqq.decay_1}
\Bigl(\frac{r(\underline{B})}{r(\overline{B})}\Bigr)^{Q} \le C(R,Q) \frac{\mu(\underline{B})}{\mu(\overline{B})} \quad \text{for all balls}
\quad \underline{B} \subset \overline{B} \quad \text{with} \quad 0< r_{\underline{B}} \le r_{\overline{B}} \le R.
\end{equation}
Furthermore, for each $R > 0$, there is $q=q(R) > 0$ such that the measure $\mu$ has the reverse volume decay property of order $q$, i.e., there exists $C(R,q) > 0$ such that
\begin{equation}
\label{eqq.decay_2}
\frac{\mu(\underline{B})}{\mu(\overline{B})} \le C(R,q) \Bigl(\frac{r(\underline{B})}{r(\overline{B})}\Bigr)^{q} \quad \text{for all balls}
\quad \underline{B} \subset \overline{B} \quad \text{with} \quad 0< r_{\underline{B}} \le r_{\overline{B}} \le R.
\end{equation}
\end{Prop}
Having at our disposal Proposition \ref{Prop.volume_decay} we let $\underline{Q}_{\mu}$ denote the \textit{infimum} of the set of all $Q$ for which \eqref{eqq.decay_1} holds. Similarly,
we let $\overline{q}_{\mu}$ denote the \textit{supremum} of all $q$ for which \eqref{eqq.decay_2} holds.

\begin{Remark}
\label{Rem.gap_between_exponents}
It is clear that $\overline{q}_{\mu} \le \underline{Q}_{\mu}$. Unfortunately, in many typical situations there is a gap between these exponents, i.e., $\overline{q}_{\mu}$ can be much smaller than
$\underline{Q}_{\mu}$. The reader can find interesting examples illustrating this phenomena in \cite{MarOrt}.
\end{Remark}

Given a set $E \subset \operatorname{X}$, for each $k \in \mathbb{Z}$, by $Z_{k}(E)$ we will denote an arbitrary $2^{-k}$-separated subset of $E$.
Furthermore, by $\mathcal{A}_{k}(E)$ we denote the corresponding index set, i.e., $Z_{k}(E):=\{z_{k,\alpha}:\alpha \in \mathcal{A}_{k}(E)\}$.
We recall the following elementary property (see \cite{T5}).
\begin{Prop}
\label{Prop.overlapping}
Given $c \geq 1$, there exists a constant $C > 0$ such that
\begin{equation}
\sup\limits_{x \in \operatorname{X}}\sup_{k \in \mathbb{N}_{0}}\sum_{\alpha \in \mathcal{A}_{k}(E)}\chi_{cB_{k,\alpha}}(x) \le C.
\end{equation}
\end{Prop}

\subsection{Regular sets}

Since the dependence of $\mu(B_{r}(x))$ on $r$ is not a power of $r$ in general, it is natural to consider
\textit{codimensional substitutions} for the usual Hausdorff contents and measures.
More precisely, following \cite{GKS,GS,Maly,MNS,SakSot}, given $\theta \geq 0$, for each set $E \subset \operatorname{X}$ and any $\delta \in (0,\infty]$, we put
\begin{equation}
\label{eq.cod.content}
\mathcal{H}_{\theta,\delta}(E):=\inf\{\sum \frac{\mu(B_{r_{i}}(x_{i}))}{(r_{i})^{\theta}}:E \subset \bigcup B_{r_{i}}(x_{i}) \text{ and } 0 < r_{i} < \delta\},
\end{equation}
where the infimum is taken over all at most countable coverings of $E$ by balls $\{B_{r_{i}}(x_{i})\}$ with radii $r_{i} \in (0,\delta)$.
Given $\delta > 0$, the mapping $\mathcal{H}_{\theta,\delta}:2^{\operatorname{X}} \to [0,+\infty]$ is called the \textit{codimension-$\theta$ Hausdorff content} at scale $\delta$. We define the \textit{codimension-$\theta$ Hausdorff measure} by the equality
\begin{equation}
\label{eqq.definition_Hausdorff_measure}
\mathcal{H}_{\theta}(E):=\lim_{\delta \to 0}\mathcal{H}_{\theta,\delta}(E).
\end{equation}

\begin{Remark}
\label{Rem.nontriviality_Hausdorff}
Clearly, given $\theta \in [0,\underline{Q}_{\mu})$ the equality $\mathcal{H}_{\theta}(\emptyset)=0$ follows from the existence of a sequence of (closed) balls $\{B_{i}\}:=\{B_{r_{i}}(x_{i})\}_{i=1}^{\infty}$
with radii $r_{i} \to 0$, $i \to \infty$
such that $\mu(B_{i})/(r_{i})^{\theta} \to 0$, $i \to \infty$. As a result, by Theorem 4.2 in \cite{Mat}, in this case $\mathcal{H}_{\theta}:2^{\operatorname{X}} \to [0,+\infty]$ is a Borel regular outer measure on $\operatorname{X}$. Obviously, the inequality $0 \le \theta < \overline{q}_{\mu}$ is sufficient for that. Unfortunately, it is far from being necessary.

The problem of finding an appropriate range of parameters for which $\mathcal{H}_{\theta}$ is a nontrivial
outer measure (i.e., there are nontrivial subsets of finite positive measure) is rather subtle and depends on a concrete structure of a given metric measure space.
The situation is completely transparent for the so-called Ahlfors $Q$-regular space, i.e., when $\mu(B_{r}(x)) \approx r^{Q}$, $r  > 0$, $x \in \operatorname{X}$ for some $Q \geq 0$ (independent on $r$ and $x$).
In this case $\mathcal{H}_{\theta}$ can be considered as a nontrivial outer measure for the range $\theta \in [0,Q)$. In the case $\theta=Q$, the measure
$\mathcal{H}_{Q}$ is a counting measure and $\mathcal{H}_{Q}(E)=+\infty$ for any infinite set $E$.
\end{Remark}

The following concept, which was actively used in the recent papers \cite{GKS,GS,Maly,MNS}, extends the well-known Ahlfors-David regularity condition to the general
metric measure settings.

\begin{Def}
\label{Def.Ahlfors_David}
Given a parameter $\theta \geq 0$, a closed set $S \subset \operatorname{X}$ is said to be \textit{Ahlfors--David codimension-$\theta$ regular} provided that there exist
constants $\varkappa_{S,1},\varkappa_{S,2} > 0$ such that
\begin{equation}
\label{eqq.Ahlfors_David}
\varkappa_{\theta,1}(S) \frac{\mu(B_{r}(x))}{r^{\theta}} \le \mathcal{H}_{\theta}(B_{r}(x) \cap S) \le \varkappa_{\kappa,2}(S) \frac{\mu(B_{r}(x))}{r^{\theta}} \quad \hbox{for all} \quad (x,r) \in S \times (0,1].
\end{equation}
The class of all Ahlfors--David codimension-$\theta$ regular sets will be denoted by $\mathcal{ADR}_{\theta}(\operatorname{X})$.
\end{Def}

\begin{Remark}
\label{Rem.nontriviality_Ahlfors_David}
It is clear from Remark \ref{Rem.nontriviality_Hausdorff} that if the metric measure space is not Ahlfors $Q$-regular
it is difficult to write down explicitly the range of $\theta \geq 0$ for which $\mathcal{ADR}_{\theta}(\operatorname{X}) \neq \emptyset$.
\end{Remark}

Given a Borel regular (outer) measure $\mathfrak{m}$ on $\operatorname{X}$ and a Borel set $S \subset \operatorname{X}$, by $\mathfrak{m}\lfloor_{S}$ we denote the restriction
of $\mathfrak{m}$ to $S$, i.e., $\mathfrak{m}\lfloor_{S}:=\mathfrak{m}(E \cap S)$ for any Borel set $E \subset \operatorname{X}$. It is well known (see, for example, Lemma 3.3.13 in \cite{HKST})
that $\mathfrak{m}\lfloor_{S}$ is a Borel regular measure on $\operatorname{X}$.

\begin{Remark}
\label{Rem.Ahlfors_doubling}
Note that, given $\theta \geq 0$ and $S \in \mathcal{ADR}_{\theta}(\operatorname{X})$, the restriction
$\mathcal{H}_{\theta} \lfloor_{S}$ of $\mathcal{H}_{\theta}$ to the set $S$ satisfies the uniformly locally doubling condition on $S$, i.e.,
for each $R > 0$,
\begin{equation}
\notag
C_{\theta}(R):=\sup_{r (0,R]}\sup_{x \in S}\frac{\mathcal{H}_{\theta}\lfloor_{S}(B_{2r}(x))}{\mathcal{H}_{\theta}\lfloor{S}(B_{r}(x))}:=\sup_{r (0,R]}\sup_{x \in S}\frac{\mathcal{H}_{\theta}(B_{2r}(x) \cap S)}{\mathcal{H}_{\theta}(B_{r}(x) \cap S)} < +\infty.
\end{equation}

Combining this observation with the Lebesgue differentiation theorem (see Section 3.4 in \cite{HKST}) we deduce that, for each $f \in L_{p}(\mathcal{H}_{\theta}\lfloor_{S})$,
there is a set $E \subset S$ with $\mathcal{H}_{\theta}(E)=0$ such that every $x \in S \setminus E$ is a Lebesgue point of $f$.
\end{Remark}

The following concept is very important in may areas of modern analysis. 
The corresponding literature is huge and we mention only the survey \cite{Shmerkin}.
Furthermore, this concept was crucial in many results concerning traces of function spaces (see \cite{TV1,T5} and the corresponding references therein).

\begin{Def}
\label{Def.porous}
Given a Borel set $S \subset \operatorname{X}$ and a parameter $\sigma \in (0,1]$, we say that a ball $B$ is \textit{$(S,\sigma)$-porous} if
there is a ball $B' \subset B \setminus S$ such that $r(B') \geq \sigma r(B)$.
Furthermore, given $r \in (0,1]$, we put $S_{r}(\sigma):=\{x \in S:B_{r}(x) \text{ is } (S,\sigma)\text{-porous}\}.$
We say that $S$ is $\sigma$-porous if $S=S_{r}(\sigma)$ for all $r \in (0,1]$.
\end{Def}

Now we summarize the basic properties of Ahlfors--David codimension-$\theta$ regular sets.
\begin{Prop}
\label{Prop.measure_plus_porosity}
Let $\theta > 0$ and $S \in \mathcal{ADR}_{\theta}(\operatorname{X})$. Then:
\begin{itemize}
\item[\((1)\)] $\mu(S)=0$;

\item[\((2)\)] there is $\sigma=\sigma(S) \in (0,1]$ depending on $\theta$, $\varkappa_{S,1}$, $\varkappa_{S,2}$ such that $S$ is $\sigma$-porous.
\end{itemize}
\end{Prop}

\begin{proof}
Without loss of generality we may assume that $S$ is bounded and hence, $\mathcal{H}_{\theta}(S) < +\infty$.
To prove (1) it is sufficient to note that, given $\delta > 0$ and a covering $\{B_{j}\}$ of $S$ with radii $r(B_{j}) \in (0,\delta)$, we clearly
have $\sum \mu(B_{j}) \le \delta^{\theta}\sum \mu(B_{j})/(r(B_{j}))^{\theta}$. Hence, $\mu(S) \le \delta^{\theta}\mathcal{H}_{\theta,\delta}(S)$.
Since $\delta > 0$ can be chosen arbitrarily small, the claim follows.

To prove (2) we repeat some standard arguments from \cite{JJKR} given there for the case of Ahlfors $Q$-regular metric measure spaces.
We fix $x \in S$ and $r \in (0,1/8]$. Given $k \in \mathbb{N}_{0}$ with $2^{-k} \le r$, consider the index set $\mathcal{C}_{k}:=\{\alpha \in \mathcal{A}_{k}(\operatorname{X}): z_{k,\alpha} \in B_{r}(x)\}$.
Assume that $B_{k,\alpha} \cap S \neq \emptyset$ for all $\alpha \in \mathcal{C}_{k}$ and, for each $\alpha \in \mathcal{C}_{k}$, 
choose $x_{k,\alpha} \in B_{k,\alpha} \cap S$. 
Clearly, we have the following inclusions
$$
B_{\frac{1}{2^{k}}}(x_{k,\alpha}) \subset 2B_{k,\alpha} \subset B_{\frac{4}{2^{k}}}(x_{k,\alpha}) \subset 8B_{k,\alpha}.
$$
Using the locally uniformly doubling property of $\mu$ and \eqref{eqq.Ahlfors_David}, we have, for each $k \in \mathbb{N}_{0}$ satisfying  $2^{-k} \le r$, for each $\alpha \in \mathcal{C}_{k}$
the following estimates (we recall that $r \in (0,1/8]$)
\begin{equation}
\notag
\begin{split}
&\frac{c_{\theta,1}}{C_{\mu}(1)}2^{k\theta}\mu(B_{k,\alpha}) \le c_{\theta,1}2^{k\theta}\mu(B_{\frac{1}{2^{k}}}(x_{k,\alpha})) \le \mathcal{H}_{\theta}(B_{\frac{1}{2^{k}}}(x_{k,\alpha}) \cap S) \le \mathcal{H}_{\theta}(2B_{k,\alpha} \cap S)\\
&\le \mathcal{H}_{\theta}(B_{\frac{4}{2^{k}}}(x_{k,\alpha}) \cap S) \le c_{\theta,2}2^{k\theta}\mu(B_{\frac{4}{2^{k}}}(x_{k,\alpha})) \le (C_{\mu}(1))^{3}c_{\theta,2}2^{k\theta}\mu(B_{k,\alpha}).
\end{split}
\end{equation}
Combining this observation with Proposition \ref{Prop.overlapping}, Remark \ref{Rem.Ahlfors_doubling} and the right-hand inequality in \eqref{eqq.Ahlfors_David} we deduce (recall again that $r \in (0,1/8]$)
\begin{equation}
\begin{split}
&2^{k\theta}\mu(B_{r}(x))\le \sum\limits_{\alpha \in \mathcal{C}_{k}}2^{k\theta}\mu(2B_{k,\alpha}) \le C \sum\limits_{\alpha \in \mathcal{C}_{k}}\mathcal{H}_{\theta}(2B_{k,\alpha} \cap S )\\
&\le C \mathcal{H}_{\theta}(B_{3r}(x) \cap S) \le C \frac{\mu(B_{r}(x))}{r^{\theta}}.
\end{split}
\end{equation}
Note that $k \in \mathbb{N}$ can be taken arbitrarily large. On the other hand, the constant $C > 0$ in the above inequality does not depend on $k$. This clearly gives a contradiction. 
Hence, there exists $N=N(\theta,C_{\mu}(1),\varkappa_{\theta,1}(S),\varkappa_{\theta,2}(S)) \in \mathbb{N}$ such that for any $k \in \mathbb{N}_{0}$ satisfying $2^{-k} \le \frac{r}{N}$ one can find a ball $B_{k,\alpha}$
whose center belongs to $B_{r}(x)$, but $B_{k,\alpha} \cap S = \emptyset$.
Taking into account that $x \in S$ and $r \in (0,\frac{1}{8}]$ was chosen arbitrarily we put $\sigma = \frac{1}{8N}$ and complete the proof.
\end{proof}

In \cite{T5} the following natural generalization of the Ahlfors--David-type regularity condition was introduced.

\begin{Def}
\label{Def.lower_content}
Given a parameter $\theta \geq 0$,
we say that a set $S \subset \operatorname{X}$
is \textit{lower codimension-$\theta$ content regular} if there exists a constant $\lambda_{S} \in (0,1]$ such that
\begin{equation}
\notag
\mathcal{H}_{\theta,r}(B_{r}(x) \cap S) \geq \lambda_{S} \frac{\mu(B_{r}(x))}{r^{\theta}} \quad \hbox{for all} \quad x \in S \quad \hbox{and all} \quad r \in (0,1].
\end{equation}
By $\mathcal{LCR}_{\theta}(\operatorname{X})$ we denote the family of all lower codimension-$\theta$ content regular subsets of $\operatorname{X}$.
\end{Def}

\begin{Remark}
\label{Rem.2.7}
One can easily show that $\mathcal{ADR}_{\theta}(\operatorname{X}) \subset \mathcal{LCR}_{\theta}(\operatorname{X})$, $\theta \geq 0$ (see Lemma 4.7 in \cite{T5}). Typically, the class $\mathcal{LCR}_{\theta}(\operatorname{X})$ is much more broad than the class $\mathcal{ADR}_{\theta}(\operatorname{X})$ (see the corresponding examples in \cite{TV1,T2}).
\end{Remark}

\begin{Remark}
\label{Rem.2.8}
It is clear that, given $0 \le \theta_{1} \le \theta_{2}$, we have $\mathcal{LCR}_{\theta_{1}}(\operatorname{X}) \subset \mathcal{LCR}_{\theta_{2}}(\operatorname{X})$.
\end{Remark}

\begin{Remark}
\label{Rem.2.9}
It is easy to see that, given $\theta \geq 0$, and arbitrary sets $S^{1}, S^{2} \in \mathcal{LCR}_{\theta}(\operatorname{X})$, the union $S=S^{1} \cup S^{2}$ also
lies in the class $\mathcal{LCR}_{\theta}(\operatorname{X})$.
\end{Remark}

\begin{Def}
\label{Def.piesewise_Ahlfors}
We say that a closed set $S \subset \operatorname{X}$ is pisewise Ahlfors--David regular
if there are numbers $0 \le \theta_{1}(S) < ... < \theta_{N}(S) < \underline{Q}_{\mu}$, $N \in \mathbb{N}$,
and sets $S^{i} \in \mathcal{ADR}_{\theta_{i}(S)}(\operatorname{X})$ such that $S=\cup_{i=1}^{N}S^{i}$.
In this case we put $\theta(S):=\theta_{N}(S)$.

By $\mathcal{PADR}(\operatorname{X})$ we denote the class
of all pisewise Ahlfors--David regular closed sets. Furthermore, given $\theta \geq 0$, we put $\mathcal{PADR}_{\operatorname{\theta}}(\operatorname{X}):=\{S \in \mathcal{PADR}(\operatorname{X}):\theta(S)=\theta\}$.
Clearly, $\mathcal{PADR}(\operatorname{X})=\cup_{\theta \geq 0}\mathcal{PADR}_{\theta}(\operatorname{X})$.
\end{Def}

\begin{Remark}
By Remarks \ref{Rem.2.7}-\ref{Rem.2.9} it is clear that, given $\theta \geq 0$, $\mathcal{PADR}_{\theta}(\operatorname{X}) \subset \mathcal{LCR}_{\theta}(\operatorname{X})$.
\end{Remark}

\begin{Remark}
\label{Rem.measure_plus_porosity}
Given a set $S \in \mathcal{PADR}_{\theta}(\operatorname{X})$, by Proposition \ref{Prop.measure_plus_porosity} we clearly have $\mu(S)=0$ provided that $\theta_{1}(S) > 0$.

Furthermore, if $S = \cup_{i=1}^{N}S^{i}$ is such that $S^{i} \in \mathcal{ADR}_{\theta_{i}(S)}(\operatorname{X})$ with $\theta_{i}(S) > 0$, $i \in \{1,...,N\}$, then $S$
is $\sigma$-porous for some $\sigma=\sigma(S) \in (0,1)$. To prove this fact we proceed as follows.

First of all we claim that if a set $S$ is $\sigma$-porous, then each ball $B$ with $r_{B} \le 1$ is $(S,2\sigma/3)$-porous.
Indeed, let $B=B_{r}(x)$ be an arbitrary ball with $r \le 1$. Consider the ball $B_{\frac{r}{3}}(x)$. If it has an empty intersection with $S$,
then we conclude. If $B_{\frac{r}{3}}(x) \cap S \neq \emptyset$, then, for each $y \in B_{\frac{r}{3}}(x) \cap S$, we have $B_{\frac{r}{3}}(x) \subset B_{\frac{2r}{3}}(y)$.
Taking into account that $B_{\frac{2r}{3}}(y)$ is $(S,\sigma)$-porous (and $\sigma \le \frac{1}{2}$ in this case) we comlete the proof of the claim.

According to Proposition \ref{Prop.measure_plus_porosity},
for each $i \in \{1,...,N\}$, the set $S^{i}$ is $\sigma_{i}:=\sigma_{i}(S^{i})$-porous for some $\sigma_{i} \in (0,1)$.
Consequently, applying the second assertion in Proposition \ref{Prop.measure_plus_porosity} $N$ times in combination with above arguments 
we see that the set $S$ is $\sigma$-porous with $\sigma=\prod_{i=1}^{N}2\sigma_{i}/3$.
\end{Remark}

\subsection{Regular sequences of measures}
Now we recall the crucial tool from \cite{T5} capable of capturing smoothness properties of functions in the trace spaces.

\begin{Def}
\label{Def.regular_sequence}
Given $\theta \geq 0$, we say that a sequence of
measures $\{\mathfrak{m}_{k}\}:=\{\mathfrak{m}_{k}\}_{k=0}^{\infty}$ on $\operatorname{X}$ is $\theta$-regular if there exists $\epsilon=\epsilon(\{\mathfrak{m}_{k}\}) \in (0,1)$ such that
the following conditions are satisfied:
\begin{itemize}
\item[\((\textbf{M}1)\)] there exists a closed nonempty set $S \subset \operatorname{X}$ such that $\operatorname{supp}\mathfrak{m}_{k}=S$ for all $k \in \mathbb{N}_{0}$;

\item[\((\textbf{M}2)\)]
there exists a constant $C_{1} > 0$ such that, for each $k \in \mathbb{N}_{0}$, $\mathfrak{m}_{k}(B_{r}(x)) \le C_{1} \frac{\mu(B_{r}(x))}{r^{\theta}}$ for all $x \in \operatorname{X}$
and all $r \in (0,\epsilon^{k}]$;

\item[\((\textbf{M}3)\)]  there exists a constant $C_{2} > 0$ such that, for each $k \in \mathbb{N}_{0}$, $\mathfrak{m}_{k}(B_{r}(x)) \geq
C_{2}\frac{\mu(B_{r}(x))}{r^{\theta}}$ for all $x \in S$ and all $r \in [\epsilon^{k},1]$;

\item[\((\textbf{M}4)\)]  $\mathfrak{m}_{k}=w_{k}\mathfrak{m}_{0}$ with $w_{k} \in L_{\infty}(\mathfrak{m}_{0})$ for every $k \in \mathbb{N}_{0}$ and, furthermore, there exists a constant
$C_{3} > 0$ such that, for each $k,j \in \mathbb{N}_{0}$, $\epsilon^{\theta j} (C_{3})^{-1} \le w_{k}(x)/w_{k+j}(x) \le C_{3}$ for $\mathfrak{m}_{0}$-a.e. $x \in S$.
\end{itemize}
Furthermore, we say that a $\theta$-regular sequence of measures $\{\mathfrak{m}_{k}\}$ is strongly $\theta$-regular if
\begin{itemize}
\item[\((\textbf{M}5)\)] for each Borel set $E \subset S$, $\varlimsup_{k \to \infty} \mathfrak{m}_{k}(B_{\epsilon^{k}}(\underline{x}) \cap E)/ \mathfrak{m}_{k}(B_{\epsilon^{k}}(\underline{x}))  > 0$ for
$\mathfrak{m}_{0}$-a.e. $\underline{x} \in E$.
\end{itemize}
\end{Def}

Given a closed set $S \subset \operatorname{X}$,
the class of all $\theta$-regular and all strongly $\theta$-regular sequences of measures $\{\mathfrak{m}_{k}\}$ satisfying $\operatorname{supp}\mathfrak{m}_{k}=S$, $k \in \mathbb{N}_{0}$ will be denoted by $\mathfrak{M}_{\theta}(S)$ and $\mathfrak{M}^{str}_{\theta}(S)$, respectively.

\begin{Remark}
Let $S$ be a closed nonempty set and let $\theta \geq 0$. It was proved in \cite{T5} that if $S \in \mathcal{LCR}_{\theta}(\operatorname{X})$, then $\mathfrak{M}^{str}_{\theta}(S) \neq \emptyset$. 
On the other hand, if $\mathfrak{M}_{\theta}(S) \neq \emptyset$, then $S \in \mathcal{LCR}_{\theta}(\operatorname{X})$.
\end{Remark}

The following proposition is an easy consequence of Definition \ref{Def.regular_sequence} (see Theorem 5.2 in \cite{T5} for the details). Roughly speaking,
it gives a some sort of doubling properties of measures $\mathfrak{m}_{k}$, $k \in \mathbb{N}_{0}$ on the corresponding scales.

\begin{Prop}
\label{Prop.doubling_property}
Let $\theta \geq 0$, $S \in \mathcal{LCR}_{\theta}(\operatorname{X})$ and $\{\mathfrak{m}_{k}\} \in \mathfrak{M}_{\theta}(\operatorname{X})$. Then, for each $c \geq 1$,
there exists a constant $C > 0$ such that, for each $k \in \mathbb{N}_{0}$,
\begin{equation}
\label{eq.doubling_type}
\frac{1}{C}\mathfrak{m}_{k}(B_{\epsilon^{k}}(y))\le \mathfrak{m}_{k}(B_{\frac{\epsilon^{k}}{c}}(y)) \le \mathfrak{m}_{k}(B_{c\epsilon^{k}}(y)) \le C  \mathfrak{m}_{k}(B_{\epsilon^{k}}(y)) \quad \text{for all} \quad y \in S.
\end{equation}
\end{Prop}

Using the above proposition we deduce the following simple but important estimate.

\begin{Prop}
\label{Prop.oscillation_comparison}
Let $\theta \geq 0$, $S \in \mathcal{LCR}_{\theta}(\operatorname{X})$ and $\{\mathfrak{m}_{k}\} \in \mathfrak{M}_{\theta}(\operatorname{X})$. Then, for each $c_{1},c_{2} \geq 1$,
there exists a constant $C > 0$ such that, for each $k \in \mathbb{N}_{0}$, the following inequality
\begin{equation}
\label{eq.oscillation_comparison}
\mathcal{E}_{\mathfrak{m}_{k}}(f,B_{c_{1}\epsilon^{k}}(x)) \le C \mathcal{E}_{\mathfrak{m}_{k}}(f,B_{c_{2}\epsilon^{k}}(y))
\end{equation}
holds for any balls $B_{\epsilon^{k}}(x)$, $B_{\epsilon^{k}}(y)$ with $x \in S$, $y \in \operatorname{X}$ satisfying $B_{c_{1}\epsilon^{k}}(x) \subset B_{c_{2}\epsilon^{k}}(y)$.
\end{Prop}

\begin{proof}
We fix a number $k \in \mathbb{N}_{0}$ and closed balls $B_{\epsilon^{k}}(x)$, $B_{\epsilon^{k}}(y)$ with $x \in S$, $y \in \operatorname{X}$ such that $B_{c_{1}\epsilon^{k}}(x) \subset B_{c_{2}\epsilon^{k}}(y)$. Clearly, $B_{c_{2}\epsilon^{k}}(y) \subset B_{2c_{2}\epsilon^{k}}(x)$. By Proposition \ref{Prop.doubling_property}, we obtain
\begin{equation}
\notag
\mathfrak{m}_{k}(B_{c_{1}\epsilon^{k}}(x)) \le \mathfrak{m}_{k}(B_{c_{2}\epsilon^{k}}(y)) \le \mathfrak{m}_{k}(B_{2c_{2}\epsilon^{k}}(x)) \le C\mathfrak{m}_{k}(B_{\epsilon^{k}}(x)) \le C\mathfrak{m}_{k}(B_{c_{1}\epsilon^{k}}(x)).
\end{equation}
Combining this estimate with \eqref{eqq.different_overagings} we get
\begin{equation}
\notag
\mathcal{OSC}_{\mathfrak{m}_{k}}(f,B_{c_{1}\epsilon^{k}}(x)) \le C \mathcal{OSC}_{\mathfrak{m}_{k}}(f,B_{c_{2}\epsilon^{k}}(y)).
\end{equation}
Hence, taking into account \eqref{eqq.averaging_comparison} we obtain the required estimate and complete the proof.
\end{proof}

\begin{Prop}
\label{Prop.special_regular_measures}
Let $\underline{\theta} \in [0,\underline{Q}_{\mu})$ and $S \in \mathcal{ADR}_{\underline{\theta}}(\operatorname{X})$. Then, for each $\theta \in [\underline{\theta},\underline{Q}_{\mu})$,
the sequence $\{2^{k(\theta-\underline{\theta})}\mathcal{H}_{\underline{\theta}}\lfloor_{S}\} \in \mathfrak{M}_{\theta}^{str}(S)$ with $\epsilon(\{2^{k(\theta-\underline{\theta})}\mathcal{H}_{\underline{\theta}}\lfloor_{S}\})=1/2$.
\end{Prop}

\begin{proof}
The fact that the sequence $\{2^{k(\theta-\underline{\theta})}\mathcal{H}_{\theta}\lfloor_{S}\}$ lies in $\mathfrak{M}_{\theta}(S)$ follows immediately from Definitions \ref{Def.Ahlfors_David} and \ref{Def.regular_sequence}. To verify condition (\textbf{M}5) in Definition \ref{Def.regular_sequence} we note that in fact a stronger condition holds. More precisely, given
a Borel set $E \subset \operatorname{X}$,
\begin{equation}
\label{eqq.1.7}
\lim\limits_{k \to \infty} \frac{\mathcal{H}_{\theta}\lfloor_{S}(B_{2^{-k}}(\underline{x}) \cap E)}{\mathcal{H}_{\theta}\lfloor_{S}(B_{2^{-k}}(\underline{x}))}  = 1 \quad \text{for} \quad
\mathcal{H}_{\theta}-\text{a.e.} \quad \underline{x} \in E.
\end{equation}
In order to verify \eqref{eqq.1.7} it is sufficient to use Remark \ref{Rem.Ahlfors_doubling} and apply the classical arguments
given in Section 3.4 in \cite{HKST}.
\end{proof}

\begin{Remark}
\label{Rem.change_variables}
By Proposition \ref{Prop.doubling_property}, given $\theta \geq 0$, $S \in \mathcal{LCR}_{\theta}(\operatorname{X})$, $\{\mathfrak{m}_{k}\} \in \mathfrak{M}_{\theta}(\operatorname{X})$ and $c \geq 1$, it follows that there exists a constant $C > 0$ such that, for each $k \in \mathbb{N}_{0}$, (see Proposition 5.3 in \cite{T5} for details)
\begin{equation}
\int\limits_{B_{c\epsilon^{k}}(z)}\frac{1}{\mathfrak{m}_{k}(B_{c\epsilon^{k}}(y))}\,d\mathfrak{m}_{k}(y) \le C \quad \text{for all} \quad z \in S.
\end{equation}

In particular, keeping in mind Proposition \ref{Prop.special_regular_measures}, we obtain that, given $\theta \geq 0$, $S \in \mathcal{ADR}_{\theta}(\operatorname{X})$ and $c \geq 1$, there is a constant $C > 0$ such that, for each $k \in \mathbb{N}_{0}$,
\begin{equation}
\int\limits_{B_{\frac{c}{2^{k}}}(z) \cap S}\frac{1}{\mathcal{H}_{\theta}(B_{\frac{c}{2^{k}}}(y) \cap S)}d\mathcal{H}_{\theta}(y) \le C.
\end{equation}
\end{Remark}

\begin{Lm}
\label{Lm.estimate_by_lp_norms}
Let $\theta \geq 0$, $S \in \mathcal{LCR}_{\theta}(\operatorname{X})$ and $\{\mathfrak{m}_{k}\} \in \mathfrak{M}_{\theta}(\operatorname{X})$. Then, for each $L \in \mathbb{N}_{0}$,
there is a constant $C > 0$ (depending on $L$) such that
\begin{equation}
\sum\limits_{k=0}^{L}\int\limits_{S}\Bigl(\mathcal{E}_{\mathfrak{m}_{k}}(f,B_{\epsilon^{k}}(x))\Bigr)^{p}\,d\mathfrak{m}_{k}(x) \le C \|f|L_{p}(\mathfrak{m}_{0})\|^{p}.
\end{equation}
\end{Lm}

\begin{proof}
By H\"older's inequality, \eqref{eqq.different_overagings} and \eqref{eqq.different_overagings} we have 
$$
\Bigl(\mathcal{E}_{\mathfrak{m}_{k}}(f,B_{\epsilon^{k}}(x))\Bigr)^{p} \le C \fint_{B_{\epsilon^{k}}(x)}|f(y)|^{p}\,d\mathfrak{m}_{k}(y).
$$
Hence, changing the order of integration and taking into account Remark \ref{Rem.change_variables} we obtain, for each $k \in \{0,...,L\}$,
\begin{equation}
\notag
\begin{split}
&\int\limits_{S}\Bigl(\mathcal{E}_{\mathfrak{m}_{k}}(f,B_{\epsilon^{k}}(x))\Bigr)^{p}\,d\mathfrak{m}_{k}(x)\\
&\le C \int\limits_{S}\Bigl(\fint\limits_{B_{\epsilon^{k}}(x)}|f(y)|^{p}\,d\mathfrak{m}_{k}(y)\Bigr)\,d\mathfrak{m}_{k}(x) \le C\int\limits_{S}|f(y)|^{p}\,d\mathfrak{m}_{k}(y).
\end{split}
\end{equation}
Summing this estimate over all $k \in \{0,...,L\}$ we get the required estimate.
\end{proof}

\subsection{Sobolev spaces}

Recall that the integrability parameter $p \in (1,\infty)$ is assumed to be \textit{fixed throughout the paper.}
Recall that there are several different approaches to Sobolev spaces on metric measure spaces (see chapter 10 in \cite{HKST} for the detailed exposition 
and \cite{T5} for the corresponding discussions).
In the present paper, we follow the approach proposed by J.~Cheeger \cite{Ch} and introduce the following definition of Sobolev spaces.
\begin{Def}
\label{Def.Cheeger_Sobolev}
The Sobolev space $W^{1}_{p}(\operatorname{X})$
is a linear space consisting of all $F \in L_{p}(\operatorname{X})$ with $\operatorname{Ch}_{p}(F) < +\infty$,
where $\operatorname{Ch}_{p}(F)$ is \textit{a Cheeger energy of $F$} defined by
\begin{equation}
\notag
\operatorname{Ch}_{p}(F):=\inf\{\varliminf\limits_{n \to \infty}\int\limits_{\operatorname{X}}(\operatorname{lip}F_{n})^{p}\,d\mu: \{F_{n}\} \subset \operatorname{LIP}(\operatorname{X}),\ F_{n} \to F \text{ in } L_{p}(\operatorname{X})\}.
\end{equation}
The space $W^{1}_{p}(\operatorname{X})$ is normed by $\|F|W_{p}^{1}(\operatorname{X})\|:=\|F|L_{p}(\operatorname{X})\|+(\operatorname{Ch}_{p}(F))^{\frac{1}{p}}.$
\end{Def}

Recall the notion of $p$-capacity $C_{p}$ (see Subsection 1.4 in \cite{BB} for the details).
It is well known that, for each element $F \in W^{1}_{p}(\operatorname{X})$, there is a Borel representative $\overline{F}$ which
has Lebesgue points everywhere on $\operatorname{X}$ except a set of $p$-capacity zero. Any such a representative will be called
a \textit{$p$-sharp representative} of $F$.

In this paper, we follow the approach to traces of Sobolev functions proposed by the author in the recent paper \cite{T5}.
Suppose we are given a Borel regular 
measure $\mathfrak{m}$ on $\operatorname{X}$ and a closed nonempty set $S \subset \operatorname{X}$ 
such that $\operatorname{supp}\mathfrak{m}=S$ and $C_{p}(S) > 0$. Assume that the measure $\mathfrak{m}$ is absolutely continuous with respect to $C_{p}$, i.e.,
for each Borel set $E \subset S$, the equality $C_{p}(E)=0$ implies the equality $\mathfrak{m}(E)=0$.
We define the $\mathfrak{m}$-trace $F|_{S}^{\mathfrak{m}}$ of any element $F \in W_{p}^{1}(\operatorname{X})$ to $S$ as the
$\mathfrak{m}$-equivalence class of the pointwise restriction $\overline{F}|_{S}$ of any $p$-sharp representative $\overline{F}$ of $F$ to the set $S$.
By $W_{p}^{1}(\operatorname{X})|^{\mathfrak{m}}_{S}$ we denote the linear space of $\mathfrak{m}$-traces of all $F \in W_{p}^{1}(\operatorname{X})$ equipped with the corresponding quotient space norm.
We also introduce the $\mathfrak{m}$-trace operator $\operatorname{Tr}|_{S}^{\mathfrak{m}}:W_{p}^{1}(\operatorname{X}) \to W_{p}^{1}(\operatorname{X})|_{S}^{\mathfrak{m}}$
which acts on $W_{p}^{1}(\operatorname{X})$ by $\operatorname{Tr}|_{S}^{\mathfrak{m}}(F):=F|_{S}^{\mathfrak{m}}$. Finally, we say that $F \in W_{p}^{1}(\operatorname{X})$ is an \textit{$\mathfrak{m}$-extension} of a given function $f: S \to \mathbb{R}$ provided that for the $\mathfrak{m}$-equivalence class $[f]_{\mathfrak{m}}$ of $f$ we have $[f]_{\mathfrak{m}}=F|^{\mathfrak{m}}_{S}$.

\subsection{Abstract criterion}
Now we briefly describe a particular case of author's recent result \cite{T5}. The detailed discussion of the concepts given in this section can be found
in \cite{T5}.

Let $\theta \geq 0$ , $S \in \mathcal{LCR}_{\theta}(\operatorname{X})$ and $\{\mathfrak{m}_{k}\} \in \mathfrak{M}^{str}_{\theta}(S)$.
Given $k \in \mathbb{N}_{0}$, for each $r > 0$, we put
\begin{equation}
\label{eqq.tricky_average}
\widetilde{\mathcal{E}}_{\mathfrak{m}_{k}}(f,B_{r}(x)):=
\begin{cases}
\mathcal{E}_{\mathfrak{m}_{k}}(f,B_{2r}(x)), \quad \text{if} \quad B_{r}(x) \cap S \neq \emptyset;\\
0, \quad \text{if} \quad B_{r}(x) \cap S = \emptyset.
\end{cases}
\end{equation}

The following concept was introduced in \cite{T5}.

\begin{Def}
Given a parameter $c \geq 1$, we say that a finite family of closed balls $\mathcal{B}=\{B_{r_{i}}(x_{i})\}_{i=1}^{N}$ is $(S,c)$-nice if the following conditions hold:
\begin{itemize}
\item[\((\textbf{F}1)\)] $B_{r_{i}}(x_{i}) \cap B_{r_{j}}(x_{j}) = \emptyset$ if $i \neq j$;

\item[\((\textbf{F}2)\)] $\max\{r_{i}:i=1,...,N\} \le 1$;

\item[\((\textbf{F}3)\)] $B_{cr_{i}}(x_{i}) \cap S \neq \emptyset$ for all $i \in \{1,...,N\}$.
\end{itemize}
We say that an $(S,c)$-nice family $\mathcal{B}=\{B_{r_{i}}(x_{i})\}_{i=1}^{N}$ is an $(S,c)$-Whitney family provided that
\begin{itemize}
\item[\((\textbf{F}4)\)] $B_{r_{i}}(x_{i}) \cap S = \emptyset$ for all $i \in \{1,...,N\}$.
\end{itemize}

\end{Def}

For the subsequent exposition we shall adopt the following notation. Given a number $r > 0$, by $k(r)$ we denote the unique
integer such that $r \in (2^{-k(r)-1},2^{-k(r)}]$.

We introduce the \textit{Brudnyi--Shvartsman-type functional}. Given $c \geq 1$, for each $f \in \cap_{k=0}^{\infty}L_{p}^{loc}(\mathfrak{m}_{k})$,
\begin{equation}
\begin{split}
\label{eq.main2}
\mathcal{BSN}_{p,\{\mathfrak{m}_{k}\},c}(f):=
\sup&\Bigl(\sum\limits_{i=1}^{N} \frac{\mu(B_{r_{i}}(x_{i}))}{r^{p}_{i}}\Bigl(\widetilde{\mathcal{E}}_{\mathfrak{m}_{k(r_{i})}}(f,B_{cr_{i}}(x_{i}))\Bigr)^{p}\Bigr)^{\frac{1}{p}},
\end{split}
\end{equation}
where the supremum is taken over all $(S,c)$-nice families of closed balls $\{B_{r_{i}}(x_{i})\}_{i=1}^{N}$.

Given $f \in \cap_{k=0}^{\infty}L_{1}^{loc}(\mathfrak{m}_{k})$, we define the \textit{$\{\mathfrak{m}_{k}\}$-Calder\'on maximal function} by the formula
\begin{equation}
\notag
f^{\sharp}_{\{\mathfrak{m}_{k}\}}(x):=\sup\limits_{r \in (0,1]}\frac{1}{r}\widetilde{\mathcal{E}}_{\mathfrak{m}_{k(r)}}(f,B_{r}(x)), \quad x \in \operatorname{X}.
\end{equation}

We recall Definition \ref{Def.porous} and define a natural analog of the Besov seminorm. Given $\sigma \in (0,1]$, we put,
for each $f \in \cap_{k=0}^{\infty}L_{1}^{loc}(\mathfrak{m}_{k})$,
\begin{equation}
\label{eq.main3}
\mathcal{BN}_{p,\{\mathfrak{m}_{k}\},\sigma}(f):=\|f^{\sharp}_{\{\mathfrak{m}_{k}\}}|L_{p}(S,\mu)\|+\Bigl(\sum\limits_{k=1}^{\infty}
\epsilon^{k(\theta-p)}\int\limits_{S_{\epsilon^{k}}(\sigma)}\Bigl(\mathcal{E}_{\mathfrak{m}_{k}}(f,B_{\epsilon^{k}}(x))\Bigr)^{p}\,d\mathfrak{m}_{k}(x)\Bigr)^{\frac{1}{p}}.
\end{equation}

Now we are ready to formulate a criterion which is part of Theorem 1.4 from \cite{T5}.

\begin{Th}
\label{Th.main}
Let $\{\mathfrak{m}_{k}\} \in \mathfrak{M}^{str}_{\theta}(S)$, $\epsilon:=\epsilon(\{\mathfrak{m}_{k}\})$, $c \geq \frac{3}{\epsilon}$ and $\sigma \in (0,\frac{\epsilon^{2}}{4c})$. Then, given $f \in L_{1}^{loc}(\{\mathfrak{m}_{k}\})$, the following conditions are equivalent:

\begin{itemize}
\item[\((i)\)] $f \in W_{p}^{1}(\operatorname{X})|_{S}^{\mathfrak{m}_{0}}$;

\item[\((ii)\)] $\operatorname{BSN}_{p,\{\mathfrak{m}_{k}\},c}(f):=\|f|L_{p}(\mathfrak{m}_{0})\|+\mathcal{BSN}_{p,\{\mathfrak{m}_{k}\},c}(f) < +\infty$;

\item[\((iii)\)] $\operatorname{BN}_{p,\{\mathfrak{m}_{k}\},\sigma}(f):=\|f|L_{p}(\mathfrak{m}_{0})\|+\mathcal{BN}_{p,\{\mathfrak{m}_{k}\},\sigma}(f) < +\infty$.

\end{itemize}
Furthermore, for each $c \geq \frac{3}{\epsilon}$ and $\sigma \in (0,\frac{\epsilon^{2}}{4c})$, for every $f \in L_{1}^{loc}(\{\mathfrak{m}_{k}\})$,
\begin{equation}
\label{eq.717}
\begin{split}
&\|f|W_{p}^{1}(\operatorname{X})|_{S}^{\mathfrak{m}_{0}}\| \approx \operatorname{BSN}_{p,\{\mathfrak{m}_{k}\},c}(f) \approx
\operatorname{BN}_{p,\{\mathfrak{m}_{k}\},\sigma}(f)
\end{split}
\end{equation}
where the equivalence constants are independent of $f$.
\end{Th}

\section{Besov spaces}

The theory of Besov spaces $\operatorname{B}^{s}_{p,q}(\operatorname{X})$, $s > 0$, $p,q \in (0,+\infty]$ on $\operatorname{X}$
is of great interest in the recent years \cite{AWYY,BPV,GBIZ,GoKoSh}.
In what follows we will work with Besov spaces defined on Ahlfors--David regular subsets of $\operatorname{X}$.
Since in this note we will not work with the whole scale of Besov spaces, we define them only for $p=q \in (1,\infty)$ and $s \in (0,1)$.

Throughout the section we put $B_{k}(x):=B_{2^{-k}}(x)$ for all $x \in \operatorname{X}$ and $k \in \mathbb{Z}$. The following proposition is an immediate
consequence of \eqref{eqq.averaging_comparison} and Remark \ref{Rem.Ahlfors_doubling}.
\begin{Prop}
\label{Prop.3.1}
Let $S \in \mathcal{ADR}_{\theta}(\operatorname{X})$ for some $\theta \in [0,\underline{Q}_{\mu})$. Given $c \geq 1$, there is a constant $C > 0$ such that,
for each $f \in L_{1}^{loc}(\mathcal{H}_{\theta}\lfloor_{S})$, for every $k \in \mathbb{N}_{0}$ and any $x,y$ with $\operatorname{d}(x,y) \le \frac{c}{2^{k}}$,
\begin{equation}
\Bigl|\fint\limits_{B_{k}(x) \cap S}f(x')\,d\mathcal{H}_{\theta}(x')-\fint\limits_{B_{k}(y) \cap S}f(y')\,d\mathcal{H}_{\theta}(y')\Bigr| \le C\mathcal{E}_{\mathcal{H}_{\theta}\lfloor_{S}}(f,B_{2^{-k}+\operatorname{d}(x,y)}(x)).
\end{equation}
\end{Prop}

The following definition is inspired by that used in \cite{SakSot}.

\begin{Def}
\label{Def.Besov}
Let $S \in \mathcal{ADR}_{\theta}(\operatorname{X})$ for some $\theta \in [0,\underline{Q}_{\mu})$. Given $s \in (0,1)$ and $p \in (1,\infty)$, a function $f \in L_{p}(\mathcal{H}_{\theta}\lfloor_{S})$ belongs
to the Besov space $\operatorname{B}^{s}_{p}(S):=\operatorname{B}^{s}_{p,p}(S)$ provided that
\begin{equation}
\label{eqq.Besov_norm_1}
\mathcal{BN}^{s}_{p}(f):=\Bigl(\sum\limits_{k=1}^{\infty}
2^{ksp}\int\limits_{S}\Bigl(\mathcal{E}_{\mathcal{H}_{\theta}\lfloor_{S}}(f,B_{k}(x))\Bigr)^{p}\,d\mathcal{H}_{\theta}(x)\Bigr)^{\frac{1}{p}} < +\infty.
\end{equation}
Furthermore, we put
\begin{equation}
\label{eqq.true_Besov_norm}
\|f|\operatorname{B}^{s}_{p}(S)\|:=\|f|L_{p}(\mathcal{H}_{\theta}\lfloor_{S})\|+\mathcal{BN}^{s}_{p}(f).
\end{equation}
\end{Def}

We also need an alternative definition of the Besov space. The corresponding characterization is given by Theorem \ref{Th.equivalent_Besov_norms} below. Note that a similar result for homogeneous Besov spaces
was obtained in \cite{GBIZ} for the whole scale of parameters $s,p,q$. In fact, based on the ideas and methods of \cite{GBIZ} one can obtain an ``inhomogeneous analogs''
of the corresponding results. However, in our particular case, more simple (and, in fact, classical) techniques work perfectly.
We present the details for the completeness of our exposition.

\begin{Th}
\label{Th.equivalent_Besov_norms}
Let $S \in \mathcal{ADR}_{\theta}(\operatorname{X})$ for some $\theta \in (0,\underline{Q}_{\mu})$. Given $s > 0$ and $p \in (1,\infty)$, a function $f \in L_{p}(\mathcal{H}_{\theta}\lfloor_{S})$ belongs
to the Besov space $\operatorname{B}^{s}_{p}(S):=\operatorname{B}^{s}_{p,p}(S)$ if and only if
\begin{equation}
\label{eqq.Besov_norm_2}
\widetilde{\mathcal{BN}}^{s}_{p}(f):=\Bigl(\sum\limits_{k=1}^{\infty}
2^{ksp}\int\limits_{S}\fint\limits_{B_{k}(x)\cap S}|f(x)-f(y)|^{p}\,d\mathcal{H}_{\theta}(y)\,d\mathcal{H}_{\theta}(x)\Bigr)^{\frac{1}{p}} < +\infty.
\end{equation}
Furthermore, there is a constant $C > 0$ such that
\begin{equation}
\label{eqq.Besov_norms_equivalence}
\frac{1}{C}\|f|\operatorname{B}^{s}_{p}(S)\| \le \|f|L_{p}(\mathcal{H}_{\theta}\lfloor_{S})\|+\widetilde{\mathcal{BN}}^{s}_{p}(f) \le C \|f|\operatorname{B}^{s}_{p}(S)\| \quad \text{for all} \quad f \in L_{p}(\mathcal{H}_{\theta}\lfloor_{S}).
\end{equation}
\end{Th}

\begin{proof}
We fix $f \in L_{p}(\mathcal{H}_{\theta}\lfloor_{S})$ and split the proof into two natural steps.

\textit{Step 1.}
Note that, given $k \in \mathbb{Z}$, $B_{k}(x) \subset B_{k-1}(y)$ for all $y \in B_{k}(x)$. Hence, using \eqref{eqq.averaging_comparison}, H\"older's inequality, and Remark \ref{Rem.Ahlfors_doubling}, it is easy to get
\begin{equation}
\label{eqq.3.5}
\Bigl(\mathcal{E}_{\mathcal{H}_{\theta}\lfloor_{S}}(f,B_{k}(x))\Bigr)^{p} \le C\fint\limits_{B_{k}(x)\cap S}\fint\limits_{B_{k-1}(y)\cap S}|f(y)-f(z)|^{p}\,d\mathcal{H}_{\theta}(y)d\mathcal{H}_{\theta}(z).
\end{equation}
We plug \eqref{eqq.3.5} into \eqref{eqq.Besov_norm_1}, change variables, and take into account Remark \ref{Rem.change_variables}. This gives
\begin{equation}
\label{eqq.3.7'}
\begin{split}
&(\mathcal{BN}^{s}_{p}(f))^{p} \le \sum\limits_{k=2}^{\infty}
2^{ksp}\int\limits_{S}\Bigl(\fint\limits_{B_{k-1}(y)\cap S}|f(y)-f(z)|^{p}\,d\mathcal{H}_{\theta}(z)\Bigr)d\mathcal{H}_{\theta}(y) \le C (\widetilde{\mathcal{BN}}^{s}_{p}(f))^{p}.
\end{split}
\end{equation}
Furthermore, by Lemma \ref{Lm.estimate_by_lp_norms} applied with $L=1$ and $\mathfrak{m}_{k}=\mathcal{H}_{\theta}\lfloor_{S}$, $k \in \mathbb{N}_{0}$ we have
\begin{equation}
\label{eqq.3.8'}
\int\limits_{S}\Bigl(\mathcal{E}_{\mathcal{H}_{\theta}\lfloor_{S}}(f,B_{1}(x))\Bigr)^{p}d\mathcal{H}_{\theta}(x) \le C\|f|L_{p}(\mathcal{H}_{\theta}\lfloor_{S})\|^{p}.
\end{equation}
As a result, combining \eqref{eqq.3.7'} and \eqref{eqq.3.8'} we get the left-hand inequality in \eqref{eqq.Besov_norms_equivalence}.

\textit{Step 2.}
We verify the right-hand inequality in \eqref{eqq.Besov_norms_equivalence}. We put $f_{B_{i}(x)}:=\int_{B_{i}(x) \cap S}f(y)\,d\mathcal{H}_{\theta}(y)$ for $x \in S$ and $i \in \mathbb{Z}$.
Having at our disposal Remark \ref{Rem.Ahlfors_doubling} we get existence of a Borel set $E \subset S$ with $\mathcal{H}_{\theta}(E)=0$ such that $f(x)=\lim_{i \to \infty}f_{B_{i}(x)}$
for all $x \in S \setminus E.$
Thus,
\begin{equation}
\label{eqq.3.7}
\begin{split}
&|f(x)-f(y)| \le \sum\limits_{i=k}^{\infty}|f_{B_{i}(x)}-f_{B_{i+1}(x)}|+|f_{B_{k}(x)}-f_{B_{k}(y)}|\\
&+\sum\limits_{i=k}^{\infty}|f_{B_{i}(y)}-f_{B_{i+1}(y)}| \quad \text{for all pairs} \quad (x,y) \in (S \setminus E) \times (S \setminus E).
\end{split}
\end{equation}
Hence, by Proposition \ref{Prop.3.1} we have
\begin{equation}
\label{eqq.3.8}
\begin{split}
&|f(x)-f(y)| \le C\sum\limits_{i=k}^{\infty}\mathcal{E}_{\mathcal{H}_{\theta}\lfloor_{S}}(f,B_{i}(x))+C\mathcal{E}_{\mathcal{H}_{\theta}\lfloor_{S}}(f,3B_{k}(x))\\
&+C\sum\limits_{i=k}^{\infty}\mathcal{E}_{\mathcal{H}_{\theta}\lfloor_{S}}(f,B_{i}(y)) \quad \text{for all pairs} \quad (x,y) \in (S \setminus E) \times (S \setminus E).
\end{split}
\end{equation}
We plug \eqref{eqq.3.8} into \eqref{eqq.Besov_norm_2}. As a result, we obtain
\begin{equation}
\label{eqq.3.10}
\begin{split}
&\Bigl(\widetilde{\mathcal{BN}}^{s}_{p}(f)\Bigr)^{p}=\int\limits_{S}\sum\limits_{k=1}^{\infty}
2^{ksp}\fint\limits_{B_{k}(x)}|f(x)-f(y)|^{p}\,d\mathcal{H}_{\theta}(y)\,d\mathcal{H}_{\theta}(x)\\
&\le C(R_{1}+R_{2}+R_{3}+R_{4}).
\end{split}
\end{equation}
By Hardy's inequality we have
\begin{equation}
\begin{split}
&R_{1}:=\int\limits_{S}\sum\limits_{k=1}^{\infty}
2^{ksp}\fint\limits_{B_{k}(x)}\Bigl(\sum\limits_{i=k}^{\infty}\mathcal{E}_{\mathcal{H}_{\theta}\lfloor_{S}}(f,B_{i}(x))\Bigr)^{p}\,d\mathcal{H}_{\theta}(y)\,d\mathcal{H}_{\theta}(x) \\
&\le C  \int\limits_{S}\sum\limits_{k=1}^{\infty}
2^{ksp}(\mathcal{E}_{\mathcal{H}_{\theta}\lfloor_{S}}(f,B_{k}(x)))^{p}\,d\mathcal{H}_{\theta}(x) \le C \|f|\operatorname{B}^{s}_{p}(S)\|^{p}.
\end{split}
\end{equation}
To estimate $R_{2}$ we change variables, use Hardy's inequality and, finally, take into account Remark \ref{Rem.change_variables}. As a result, we get
\begin{equation}
\begin{split}
&R_{2}:=\int\limits_{S}\sum\limits_{k=1}^{\infty}
2^{ksp}\fint\limits_{B_{k}(x)}\Bigl(\sum\limits_{i=k}^{\infty}\mathcal{E}_{\mathcal{H}_{\theta}\lfloor_{S}}(f,B_{i}(y))\Bigr)^{p}\,d\mathcal{H}_{\theta}(y)\,d\mathcal{H}_{\theta}(x) \\
&\le C\int\limits_{S}\sum\limits_{k=1}^{\infty}
2^{ksp}\Bigl(\sum\limits_{i=k}^{\infty}\mathcal{E}_{\mathcal{H}_{\theta}\lfloor_{S}}(f,B_{i}(y))\Bigr)^{p}\,d\mathcal{H}_{\theta}(y)
\le C  \|f|\operatorname{B}^{s}_{p}(S)\|^{p}.
\end{split}
\end{equation}
Since $3B_{k}(x) \subset B_{k-2}(x)$, a combination of Proposition \ref{Prop.oscillation_comparison} and Proposition \ref{Prop.special_regular_measures} (with $\theta=\underline{\theta}$ and $\mathfrak{m}_{k}=\mathcal{H}_{\theta}\lfloor_{S}$, $k \in \mathbb{N}_{0}$) gives
\begin{equation}
\begin{split}
&R_{3}:=\int\limits_{S}\sum\limits_{k=3}^{\infty}
2^{ksp}\Bigl(\mathcal{E}_{\mathcal{H}_{\theta}\lfloor_{S}}(f,3B_{k}(x))\Bigr)^{p}\,d\mathcal{H}_{\theta}(x) \le \sum\limits_{k=3}^{\infty}\int\limits_{S}
2^{ksp}\Bigl(\mathcal{E}_{\mathcal{H}_{\theta}\lfloor_{S}}(f,B_{k-2}(x))\Bigr)^{p}\,d\mathcal{H}_{\theta}(x)\\
&\le 2^{2\theta s}\sum\limits_{k=1}^{\infty}\int\limits_{S}
2^{ksp}\Bigl(\mathcal{E}_{\mathcal{H}_{\theta}\lfloor_{S}}(f,B_{k}(x))\Bigr)^{p}\,d\mathcal{H}_{\theta}(x).
\end{split}
\end{equation}
Finally, changing variables and using Remark \ref{Rem.change_variables} it is easy to deduce
\begin{equation}
\label{eqq.3.13}
R_{4}:=\int\limits_{S}\sum\limits_{k=1}^{2}
2^{ksp}\fint\limits_{B_{k}(x) \cap S}|f(x)-f(y)|^{p}\,d\mathcal{H}_{\theta}(y)\,d\mathcal{H}_{\theta}(x) \le C\int\limits_{S}|f(x)|^{p}\,d\mathcal{H}_{\theta}(x).
\end{equation}

Combining \eqref{eqq.3.10}--\eqref{eqq.3.13} we obtain the right-hand inequality in \eqref{eqq.Besov_norms_equivalence}.

The proof is complete.
\end{proof}

\section{Main results}

Throughout this section we fix $\underline{\theta} \in [0,\min\{p,\underline{Q}_{\mu}\})$, $\theta \in [\underline{\theta},p)$,
and a set $S \in \mathcal{PADR}_{\theta}(\operatorname{X})$ such that $S=\cup_{i=1}^{N}S^{i}$ for some $N \in \mathbb{N}$, $N \geq 2$
and $S^{i} \in \mathcal{ADR}_{\theta_{i}}(\operatorname{X})$, $0 \le \theta_{1} < ... < \theta_{N}=\underline{\theta} < p$.
For each $k \in \mathbb{N}_{0}$, we put
\begin{equation}
\label{eqq.concrete_measure}
\mathfrak{m}_{k}:=\sum\limits_{i=1}^{N}2^{k(\theta-\theta_{i})}\mathcal{H}_{\theta_{i}}\lfloor_{S^{i}}.
\end{equation}

First of all, we make a simple observation which follows immediately from \eqref{eqq.concrete_measure}.
\begin{Prop}
\label{Prop.different_lp_norms}
A function $f: S \to \mathbb{R}$ belongs to $L^{loc}_{p}(\mathfrak{m}_{0})$ if and only if $f \in \cap_{i=1}^{N} L_{1}^{loc}(\mathcal{H}_{\theta_{i}}\lfloor_{S^{i}})$, and, furthermore (we set $p':=\frac{p}{p-1}$),
\begin{equation}
\|f|L_{p}(\mathfrak{m}_{0})\| \le \sum\limits_{i=1}^{N}\|f|L_{p}(\mathcal{H}_{\theta_{i}}\lfloor_{S^{i}})\| \le N^{\frac{1}{p'}}\|f|L_{p}(\mathfrak{m}_{0})\|.
\end{equation}
\end{Prop}

Furthermore, given $k \in \mathbb{N}_{0}$, we will occasionally compare the measures $\mathfrak{m}_{k}\lfloor_{S^{i}}$ and $2^{k(\theta-\theta_{i})}\mathcal{H}_{\theta_{i}}\lfloor_{S^{i}}$.
\begin{Prop}
\label{Prop.comparison_of_measures}
Given $c \geq 1$, there is a constant $C > 0$ such that, for each $i \in \{1,...,N\}$, the following holds. If $\underline{x} \in \operatorname{X}$ and
$x \in S^{i}$ are such that $B_{k}(x) \subset cB_{k}(\underline{x})$, then
\begin{equation}
\label{eqq.comparison_of_measures}
2^{k(\theta-\theta_{i})}\mathcal{H}_{\theta_{i}}(cB_{k}(\underline{x}) \cap S^{i}) \le \mathfrak{m}_{k}(cB_{k}(\underline{x})) \le C 2^{k(\theta-\theta_{i})}\mathcal{H}_{\theta_{i}}(B_{k}(x) \cap S^{i}).
\end{equation}
\end{Prop}

\begin{proof}
The left inequality in \eqref{eqq.comparison_of_measures} is an immediate consequence of \eqref{eqq.concrete_measure}. To prove the right inequality in \eqref{eqq.comparison_of_measures},
note that $cB_{k}(\underline{x}) \subset 2cB_{k}(x)$. Hence, using Proposition \ref{Prop.doubling_property}, condition (\textbf{M}2) of Definition \ref{Def.regular_sequence}
and Definition \ref{Def.Ahlfors_David}, we obtain
\begin{equation}
\notag
\mathfrak{m}_{k}(cB_{k}(\underline{x})) \le \mathfrak{m}_{k}(2cB_{k}(x)) \le C \mathfrak{m}_{k}(B_{k}(x)) \le C 2^{k\theta}\mu(B_{k}(x)) \le C 2^{k(\theta-\theta_{i})}\mathcal{H}_{\theta_{i}}(B_{k}(x) \cap S^{i}).
\end{equation}
This completes the proof.
\end{proof}

Given $k \in \mathbb{Z}$, we will use notation $B_{k}(x):=B_{2^{-k}}(x)$. For each $k \in \mathbb{Z}$, we fix
a maximal $2^{-k}$ separated subset $Z_{k}(S):=\{z_{k,\alpha}: \alpha \in \mathcal{A}_{k}(S)\}$ of $S$. Furthermore, for each $k \in \mathbb{Z}$ and any $\alpha \in \mathcal{A}_{k}(S)$, we put $B_{k,\alpha}:=B_{k}(z_{k,\alpha})$.
Given $k \in \mathbb{Z}$ and $i,j \in \{1,...,N\}$, we put
\begin{equation}
\label{eqq.s_set}
S^{i,j}_{k}:=\{x \in S^{i}:B_{2^{-k}}(x) \cap S^{j} \neq \emptyset\}.
\end{equation}
It is easy to see that $S^{i,j}_{k+1} \subset S^{i,j}_{k}$ for all $k \in \mathbb{Z}$.
Finally, for each $k \in \mathbb{Z}$, we define
\begin{equation}
\label{eqq.sigma_set}
\Sigma^{i,j}_{k}:=\{(y,z) \in S^{i} \times S^{j}: \operatorname{d}(y,z) \le 2^{-k}\}.
\end{equation}

Given $i \in \{1,...,N\}$ and $f \in L_{1}^{loc}(\mathcal{H}_{\theta_{i}}\lfloor_{S^{i}})$, for each $k \in \mathbb{Z}$, we introduce the following averaging
\begin{equation}
\label{eqq.averaging_1}
A^{i}_{k}(f)(x):=\fint\limits_{B_{k}(x) \cap S^{i}}f(x')\,d\mathcal{H}_{\theta_{i}}(x'), \quad x \in S^{i}.
\end{equation}
More generally, given $i,j \in \{1,...,N\}$ and $f \in L_{1}^{loc}(\mathcal{H}_{\theta_{i}}\lfloor_{S^{i}}) \cap L_{1}^{loc}(\mathcal{H}_{\theta_{j}}\lfloor_{S^{j}})$,
for each $k \in \mathbb{Z}$, we introduce the double averaging
\begin{equation}
\label{eqq.averaging_2}
A^{i,j}_{k}(f)(y,z):=\fint\limits_{B_{k}(y) \cap S^{i}}\fint\limits_{B_{k}(z) \cap S^{j}}|f(y')-f(z')|\,d\mathcal{H}_{\theta_{i}}(y')d\mathcal{H}_{\theta_{j}}(z'), \quad (y,z) \in S^{i}\times S^{j}.
\end{equation}

The following two lemmas will be keystone tools for us.

\begin{Lm}
\label{Lm.estimate_double_averaging_1}
Given $c \geq 1$, there is a constant $C > 0$ such that, for each $i,j \in \{1,...,N\}$, for each $k \in \mathbb{N}_{0}$, for any $(y,z) \in \Sigma^{i,j}_{k}$, the
following inequality
\begin{equation}
\label{eqq.4.7'}
A^{i,j}_{k}(f)(y,z) \le C \mathcal{E}_{\mathfrak{m}_{k}}(f,B).
\end{equation}
holds for every ball
$B=cB_{k}(\underline{x})$, $\underline{x} \in \operatorname{X}$ satisfying $B \supset B_{k}(y)$ and $B \supset B_{k}(z)$.
\end{Lm}

\begin{proof}
Clearly, $B \subset 2cB_{k}(y)$ and $B \subset 2cB_{k}(z)$. Hence, by \eqref{eqq.concrete_measure} we have
\begin{equation}
\begin{split}
&2^{k(2\theta-\theta_{i}-\theta_{j})}\int\limits_{B_{k}(y) \cap S^{i}}\int\limits_{B_{k}(z) \cap S^{j}}|f(y')-f(z')|\,d\mathcal{H}_{\theta_{i}}(y')d\mathcal{H}_{\theta_{j}}(z')\\
&\le \int\limits_{B}\int\limits_{B}|f(y')-f(z')|\,d\mathfrak{m}_{k}(y')d\mathfrak{m}_{k}(z').
\end{split}
\end{equation}
It remains to apply Proposition \ref{Prop.comparison_of_measures} and take into account \eqref{eqq.averaging_comparison}.
\end{proof}

Having at our disposal the above lemma, we can establish a useful estimate. We recall Remark \ref{Rem.measure_plus_porosity}.
\begin{Prop}
\label{Prop.estimate_simplebesov_by_modifiedbesov}
If $\theta_{i}(S) > 0$, then, for each $\sigma \in (0,\sigma(S)]$, there exists a constant $C > 0$ such that
\begin{equation}
\label{eqq.4.8'}
\|f|\operatorname{B}^{1-\frac{\theta_{i}}{p}}_{p}(S^{i})\| \le C \operatorname{BN}_{p,\{\mathfrak{m}_{k}\},\sigma}(f) \quad \text{for all} \quad f \in L_{1}^{loc}(\mathfrak{m}_{0}).
\end{equation}
\end{Prop}

\begin{proof}
We apply Lemma \ref{Lm.estimate_double_averaging_1} with $i=j$, $c=1$ and $y=z=x$. This gives existence of a constant $C > 0$ such that for each $f \in L_{1}^{loc}(\mathfrak{m}_{0})$,
\begin{equation}
\notag
\mathcal{E}_{\mathcal{H}_{\theta_{i}}\lfloor_{S^{i}}}(f,B_{k}(x)) \le C \mathcal{E}_{\mathfrak{m}_{k}}(f,B_{k}(x)), \quad k\in \mathbb{N}_{0}, \quad x \in S^{i}.
\end{equation}
This inequality immediately implies existence of a constant $C > 0$ such that
\begin{equation}
\begin{split}
\label{eqq.4.9'}
&\sum\limits_{k=1}^{\infty}2^{kp(1-\frac{\theta_{i}}{p})}\int\limits_{S^{i}}(\mathcal{E}_{\mathcal{H}_{\theta_{i}}\lfloor_{S^{i}}}(f,B_{k}(x)))^{p}\,d\mathcal{H}_{\theta_{i}}(x)\\
&\le C  \sum\limits_{k=1}^{\infty}2^{k(p-\theta)}\int\limits_{S^{i}}(\mathcal{E}_{\mathfrak{m}_{k}}(f,B_{k}(x)))^{p}\,d\mathfrak{m}_{k}(x), \quad f \in L_{1}^{loc}(\mathfrak{m}_{0}).
\end{split}
\end{equation}
Since $S^{i}$ is $\sigma$-porous (we recall Proposition \ref{Prop.measure_plus_porosity}), there is $C > 0$ such that, for each $f \in L_{1}^{loc}(\mathfrak{m}_{0})$, the right-hand of \eqref{eqq.4.9'} can be estimated from above by $C\mathcal{BN}_{p,\{\mathfrak{m}_{k}\},\sigma}(f)$.
Combining this observation with Proposition \ref{Prop.different_lp_norms} and Definition \ref{Def.Besov} we get the required estimate completing the
proof.
\end{proof}

For the next results we establish the following combinatorial assertion.
\begin{Prop}
\label{Prop.combinatorial_trick}
Let $k \in \mathbb{N}_{0}$, $\underline{x} \in \operatorname{X}$ and $c \geq 1$ be such that $cB_{k}(\underline{x}) \cap S \neq \emptyset$.
Then there is an index set $\mathcal{I} \subset \{1,...,N\}$ and a number $\overline{i} \in \{1,...,N+1\}$ such that the following holds.
For each $i \in \mathcal{I}$ there is a point $x_{i} \in S^{i}$ such that $B_{k}(x_{i}) \subset (c+\overline{i})B_{k}(\underline{x})$ and, furthermore,
$(c+\overline{i})B_{k}(\underline{x}) \cap S^{j} = \emptyset$ for all $j \in \{1,...,N\} \setminus \mathcal{I}$.
\end{Prop}

\begin{proof}
Given $l \in \{0,...,N\}$, let $\mathcal{I}^{l} \subset \{1,...,N\}$ be such that $S^{i} \cap (c+l)B_{k}(\underline{x}) \neq \emptyset$ for all $i \in \mathcal{I}^{l}$.
Now we consider the ball $(c+l+1)B_{k}(\underline{x})$. Clearly, for each $i \in \mathcal{I}^{l}$ there is $x_{i} \in S^{i}$ such that
$B_{k}(x_{i}) \subset (c+l+1)B_{k}(\underline{x})$. If $(c+l+1)B_{k}(\underline{x}) \cap S^{j} = \emptyset$ for all $j \in \{1,...,N\} \setminus \mathcal{I}^{l}$ then
we stop and put $\mathcal{I}:=\mathcal{I}^{l}$ and $\overline{i}:=l+1$. Otherwise, we repeat this procedure with $l$ replaced by $l+1$.

Clearly, since $N < +\infty$, starting from $l=0$, we find $i \in \{0,...,N\}$ such that the above procedure stops after $i$ steps. This proves the claim.
\end{proof}

\begin{Lm}
\label{Lm.estimate_double_averaging_2}
Given $c \geq 1$, there is a constant $C > 0$ such that if a ball $B=cB_{k}(\underline{x})$ and an index set $\mathcal{I} \subset \{1,...,N\}$ are
such that $cB_{k}(x) \supset B_{k}(x_{i})$ with $x_{i} \in S^{i}$ for all $i \in \mathcal{I}$ and $cB_{k}(x) \cap S^{j} = \emptyset$ for all $j \in \{1,...,N\} \setminus \mathcal{I}$, then
\begin{equation}
\label{eqq.4.9}
\begin{split}
&\mathcal{E}_{\mathfrak{m}_{k}}(f,B) \le C \Bigl(\sum\limits_{i \in \mathcal{I}}\mathcal{E}_{\mathcal{H}_{\theta_{i}}\lfloor_{S^{i}}}(f,B)+\sum\limits_{\substack{i,j \in \mathcal{I}\\ i \neq j}}\fint\limits_{B \cap S^{i}}\fint\limits_{B \cap S^{j}}|f(y')-f(z')|\,d\mathcal{H}_{\theta_{i}}(y')d\mathcal{H}_{\theta_{j}}(z')\Bigr).
\end{split}
\end{equation}
\end{Lm}

\begin{proof}
Using \eqref{eqq.concrete_measure} by elementary combinatorial observations we readily have
\begin{equation}
\notag
\begin{split}
&\int\limits_{B}\int\limits_{B}|f(y)-f(z)|\,d\mathfrak{m}_{k}(y)d\mathfrak{m}_{k}(z)\\
&= \sum\limits_{i,j \in \mathcal{I}}2^{k(\theta-\theta_{i})}2^{k(\theta-\theta_{j})}\int\limits_{B \cap S^{i}}\int\limits_{B \cap S^{j}}|f(y)-f(z)|\,d\mathcal{H}_{\theta_{i}}(y)d\mathcal{H}_{\theta_{j}}(z).
\end{split}
\end{equation}
Hence, taking into account Proposition \ref{Prop.comparison_of_measures} and \eqref{eqq.different_overagings}, we have
\begin{equation}
\notag
\begin{split}
&\mathcal{OSC}_{\mathfrak{m}_{k}}(f,B) \le C \Bigl(\sum\limits_{i \in \mathcal{I}}\mathcal{OSC}_{\mathcal{H}_{\theta_{i}}\lfloor_{S^{i}}}(f,B)+\sum\limits_{\substack{i,j \in \mathcal{I}\\ i \neq j}}\fint\limits_{B \cap S^{i}}\fint\limits_{B \cap S^{j}}|f(y')-f(z')|\,d\mathcal{H}_{\theta_{i}}(y')d\mathcal{H}_{\theta_{j}}(z')\Bigr).
\end{split}
\end{equation}
Finally, taking into account \eqref{eqq.averaging_comparison} we conclude.
\end{proof}

Given $k \in \mathbb{Z}$, we define the \textit{weight function} $\operatorname{w}_{k}:\operatorname{X} \times \operatorname{X} \to [0,+\infty)$ by the equality
\begin{equation}
\label{eqq.weight}
\operatorname{w}_{k}(y,z):=\frac{1}{\sqrt{\mu(B_{k}(y))}\sqrt{\mu(B_{k}(z))}}, \quad (y,z) \in \operatorname{X} \times \operatorname{X}.
\end{equation}
An immediate consequence of \eqref{eqq.doubling} is that, for each $\underline{k} \in \mathbb{Z}$, there is a constant $C > 0$ such that, for any $k \geq \underline{k}$,
\begin{equation}
\label{eqq.weights}
\operatorname{w}_{k}(y,z) \le C\operatorname{w}_{k-1}(y,z), \quad \text{for all} \quad (y,z) \in \operatorname{X} \times \operatorname{X}.
\end{equation}
The following simple proposition shows that the weights $\operatorname{w}_{k}$, $k \in \mathbb{N}_{0}$ can not oscillate wildly.

\begin{Prop}
\label{Prop.weight_ineq}
Given $c > 0$, there is a constant $C > 0$ such that, for each $k \in \mathbb{N}_{0}$, for every pair $(y,z) \in \operatorname{X} \times \operatorname{X}$,
\begin{equation}
\label{eqq.weight_ineq_1}
\frac{1}{C}\operatorname{w}_{k}(y',z') \le \operatorname{w}_{k}(y,z) \le C \operatorname{w}_{k}(y',z') \quad \text{for all} \quad (y',z') \in cB_{k}(y)\times cB_{k}(z).
\end{equation}
\end{Prop}
\begin{proof}
Note that $B_{k}(y) \subset 2cB_{k}(y')$ for every $y' \in cB_{k}(y)$. Hence, by \eqref{eqq.doubling} we easily get the left-hand inequality in \eqref{eqq.weight_ineq_1}.
Changing the role of $(y,z)$ and $(y',z')$ we get the right-hand inequality in \eqref{eqq.weight_ineq_1}.
\end{proof}

\begin{Prop}
\label{Prop.integral_of_weights}
Given $i,j \in \{1,...,N\}$, there is $C > 0$ such that, for each $k \in \mathbb{N}_{0}$,
\begin{equation}
\notag
\int\limits_{B_{k}(y) \cap S^{j}}\operatorname{w}_{k}(y,z)\,d\mathcal{H}_{\theta_{j}}(z) \le C 2^{k\theta_{j}}, \quad \text{for all} \quad y \in S^{i,j}_{k}.
\end{equation}
\end{Prop}

\begin{proof}
Using Proposition \ref{Prop.weight_ineq} with $y'=z'=z$, Remark \ref{Rem.change_variables} (with $\theta=\theta_{j}$) and  Definition \ref{Def.Ahlfors_David} we have
\begin{equation}
\notag
\int\limits_{B_{k}(y) \cap S^{j}}\operatorname{w}_{k}(y,z)\,d\mathcal{H}_{\theta_{j}}(z) \le C 2^{k\theta_{j}} \int\limits_{B_{k}(y) \cap S^{j}}\frac{1}{\mathcal{H}_{\theta_{j}}(B_{k}(z)\cap S^{j})}\,d\mathcal{H}_{\theta_{j}}(z) \le C 2^{k\theta_{j}}.
\end{equation}
This completes the proof.
\end{proof}

Now we are ready to introduce the following \textit{gluing functionals.} Given $f \in \cap_{i=1}^{N}L_{p}(\mathcal{H}_{\theta_{i}}\lfloor_{S^{i}})$,
\begin{equation}
\label{eqq.gluing_functionals}
\begin{split}
&\mathcal{GL}^{(1)}_{p}(f):=\Bigl(\sum\limits_{i \neq j}\sum\limits_{k=1}^{\infty}2^{k(p-\theta_{i}-\theta_{j})}\iint\limits_{\Sigma^{i,j}_{k}}\operatorname{w}_{k}(y,z)|f(y)-f(z)|^{p}\,d\mathcal{H}_{\theta_{i}}(y)d\mathcal{H}_{\theta_{j}}(z)\Bigr)^{\frac{1}{p}};\\
&\mathcal{GL}^{(2)}_{p}(f):=\Bigl(\sum\limits_{i \neq j}\sum\limits_{k=1}^{\infty}2^{k(p-\theta_{i}-\theta_{j})}\iint\limits_{\Sigma^{i,j}_{k}}\operatorname{w}_{k}(y,z)|A_{k}^{i}(f)(y)-A^{j}_{k}f(z)|^{p}\,d\mathcal{H}_{\theta_{i}}(y)d\mathcal{H}_{\theta_{j}}(z)\Bigr)^{\frac{1}{p}};\\
&\mathcal{GL}^{(3)}_{p}(f):=\Bigl(\sum\limits_{i \neq j}\sum\limits_{k=1}^{\infty}2^{k(p-\theta_{i}-\theta_{j})}\iint\limits_{\Sigma^{i,j}_{k}}\operatorname{w}_{k}(y,z)(A^{i,j}_{k}(f)(y,z))^{p}\,d\mathcal{H}_{\theta_{i}}(y)d\mathcal{H}_{\theta_{j}}(z)\Bigr)^{\frac{1}{p}}.
\end{split}
\end{equation}

\begin{Remark}
\label{Rem.gluing_ineq_1}
Given $f \in \cap_{i=1}^{N}L_{p}(\mathcal{H}_{\theta_{i}}\lfloor_{S^{i}})$, it is clear that $|A_{k}^{i}(f)(y)-A^{j}_{k}(f)(z)| \le A^{i,j}_{k}(f)(y,z)$, for each $i,j \in \{1,...,N\}$, for any pair $(y,z) \in S^{i} \times S^{j}$. Hence, $\mathcal{GL}^{(2)}_{p}(f) \le \mathcal{GL}^{(3)}_{p}(f)$.
\end{Remark}

\subsection{Simple case} In this subsection, we consider the ``simple case'' when $\theta_{1} > 0$. By a ``simplicity'' we mean that
the resulting trace space will be some sort of a mixture of function \textit{``spaces of the same nature''},
i.e., the Besov spaces with different smoothness exponents. Furthermore, the corresponding trace norm will be composed of
Besov norms in combination with the special
gluing conditions between pieces of different codimension.

Since $\theta_{i} > 0$ for all $i \in \{1,...,N\}$, we can derive interesting inequalities relating different gluing functionals.

\begin{Lm}
\label{Lm.1}
There is a constant $C > 0$ such that
\begin{equation}
\begin{split}
\label{eqq.gluing_ineq_2}
&\mathcal{GL}^{(3)}_{p}(f) \le C\Bigl(\mathcal{GL}^{(1)}_{p}(f)+\sum\limits_{i=1}^{N}\|f|L_{p}(\mathcal{H}_{\theta_{i}}\lfloor_{S^{i}})\|\Bigr) \quad \text{for all} \quad f \in \cap_{i=1}^{N}L_{p}(\mathcal{H}_{\theta_{i}}\lfloor_{S^{i}}).
\end{split}
\end{equation}
\end{Lm}

\begin{proof}
We fix $f \in \cap_{i=1}^{N}L_{p}(\mathcal{H}_{\theta_{i}}\lfloor_{S^{i}})$. Given $i,j \in \{1,...,N\}$, by Proposition \ref{Prop.weight_ineq} and H\"older's inequality, for every $k \in \mathbb{N}$, we have
\begin{equation}
\notag
\operatorname{w}_{k}(y,z)(A^{i,j}_{k}(f)(y,z))^{p} \le C \fint\limits_{B_{k}(y) \cap S^{i}}\fint\limits_{B_{k}(z) \cap S^{j}}\operatorname{w}_{k}(y',z')|f(y')-f(z')|^{p}\,d\mathcal{H}_{\theta_{i}}(y')d\mathcal{H}_{\theta_{j}}(z').
\end{equation}
Clearly, $(y',z') \in \Sigma^{i,j}_{k-1}$ provided that $(y,z) \in \Sigma_{k}^{i,j}$ and $y' \in B_{k}(y) \cap S^{i}$, $z' \in B_{k}(z) \cap S^{j}$.
Hence, changing variables in the integral, using Remark \ref{Rem.change_variables} and \eqref{eqq.weights}, we get
\begin{equation}
\notag
\begin{split}
&\sum\limits_{k=2}^{\infty}\sum\limits_{i\neq j}\iint\limits_{\Sigma^{i,j}_{k}}\operatorname{w}_{k}(y,z)(A^{i,j}_{k}(f)(y,z))^{p}\,d\mathcal{H}_{\theta_{i}}(y)d\mathcal{H}_{\theta_{j}}(z)\\
&\le C \sum\limits_{k=2}^{\infty}\sum\limits_{i\neq j}\iint\limits_{\Sigma^{i,j}_{k-1}}\operatorname{w}_{k-1}(y',z')|f(y')-f(z')|^{p}\,d\mathcal{H}_{\theta_{i}}(y')d\mathcal{H}_{\theta_{j}}(z') \le C \Bigl(\mathcal{GL}^{(1)}_{p}(f)\Bigr)^{p}.
\end{split}
\end{equation}
It remains to note that by Propositions \ref{Prop.weight_ineq}, \ref{Prop.integral_of_weights} and H\"older's inequality it is easy to see that
\begin{equation}
\notag
\begin{split}
&\sum\limits_{i,j=1}^{N}\iint\limits_{\Sigma^{i,j}_{1}}\operatorname{w}_{1}(y,z)(A^{i,j}_{1}(f)(y,z))^{p}\,d\mathcal{H}_{\theta_{i}}(y)d\mathcal{H}_{\theta_{j}}(z)\\
&\le C \sum\limits_{i=1}^{N} \int\limits_{S^{i}}\fint\limits_{B_{k}(y)\cap S^{i}}|f(y')|^{p}\,d\mathcal{H}_{\theta_{i}}(y')\,d\mathcal{H}_{\theta_{i}}(y) \le C \sum\limits_{i=1}^{N}\|f|L_{p}(\mathcal{H}_{\theta_{i}}\lfloor_{S^{i}})\|^{p}.
\end{split}
\end{equation}
Collecting the above estimates we complete the proof.
\end{proof}

\begin{Lm}
\label{Lm.2}
There is a constant $C > 0$ such that
\begin{equation}
\label{eqq.4.11'}
\mathcal{GL}^{(1)}_{p}(f) \le C\Bigl(\mathcal{GL}^{(2)}_{p}(f)+\sum\limits_{i=1}^{N}\|f|\operatorname{B}^{1-\frac{\theta_{i}}{p}}_{p}(S^{i})\|\Bigr) \quad \text{for all} \quad f \in \cap_{i=1}^{N}L_{p}(\mathcal{H}_{\theta_{i}}\lfloor_{S^{i}}).
\end{equation}
\end{Lm}

\begin{proof}
By the triangle inequality and H\"older's inequality for sums, given $k \in \mathbb{N}_{0}$, for each pair $(y,z) \in \Sigma^{i,j}_{k}$,
\begin{equation}
\notag
\begin{split}
&\operatorname{w}_{k}(y,z)|f(y)-f(z)|^{p} \le 3^{p-1}\operatorname{w}_{k}(y,z)\Bigl(|f(y)-A^{i}_{k}(f)(y)|^{p}\\
&+|f(z)-A^{j}_{k}(f)(z)|^{p}+|A^{i}_{k}(f)(y)-A^{j}_{k}(f)(z)|^{p}\Bigr).
\end{split}
\end{equation}
Hence, we have
\begin{equation}
\label{eqq.4.15}
\begin{split}
\iint\limits_{\Sigma^{i,j}_{k}}\operatorname{w}_{k}(y,z)|f(y)-f(z)|^{p}\,d\mathcal{H}_{\theta_{i}}(y)d\mathcal{H}_{\theta_{j}}(z) \le C \Bigl(J^{i,j}_{k}(1)+J^{i,j}_{k}(2)+J^{i,j}_{k}(3)\Bigr),
\end{split}
\end{equation}
where we set
\begin{equation}
\label{eqq.4.16}
\begin{split}
&J^{i,j}_{k}(1):=\iint\limits_{\Sigma^{i,j}_{k}}\operatorname{w}_{k}(y,z)|f(y)-A^{i}_{k}(f)(y)|^{p}\,d\mathcal{H}_{\theta_{i}}(y)d\mathcal{H}_{\theta_{j}}(z),\\
&J^{i,j}_{k}(2):=\iint\limits_{\Sigma^{i,j}_{k}}\operatorname{w}_{k}(y,z)|f(z)-A^{j}_{k}(f)(z)|^{p}\,d\mathcal{H}_{\theta_{i}}(y)d\mathcal{H}_{\theta_{j}}(z),\\
&J^{i,j}_{k}(3):=\iint\limits_{\Sigma^{i,j}_{k}}\operatorname{w}_{k}(y,z)|A^{i}_{k}(f)(y)-A^{j}_{k}(f)(z)|^{p}\,d\mathcal{H}_{\theta_{i}}(y)d\mathcal{H}_{\theta_{j}}(z).
\end{split}
\end{equation}

Thus, using Proposition \ref{Prop.integral_of_weights}, H\"older's inequality and taking into account \eqref{eqq.averaging_1} we deduce
\begin{equation}
\label{eqq.4.17}
\begin{split}
&J^{i,j}_{k}(1) \le C 2^{k\theta_{j}} \int\limits_{S^{i,j}_{k}}|f(y)-A^{i}_{k}(f)(y)|^{p}\,d\mathcal{H}_{\theta_{i}}(y) \le C 2^{k\theta_{j}} \int\limits_{S^{i,j}_{k}}\fint\limits_{B_{k}(y)\cap S^{i}}|f(y)-f(y')|^{p}\,d\mathcal{H}_{\theta_{i}}(y')\,d\mathcal{H}_{\theta_{i}}(y)\\
& \le C 2^{k\theta_{j}} \int\limits_{S^{i}}\fint\limits_{B_{k}(y)\cap S^{i}}|f(y)-f(y')|^{p}\,d\mathcal{H}_{\theta_{i}}(y')\,d\mathcal{H}_{\theta_{i}}(y).
\end{split}
\end{equation}
Similar arguments give
\begin{equation}
\label{eqq.4.18}
\begin{split}
J^{i,j}_{k}(2) \le  C 2^{k\theta_{i}} \int\limits_{S^{j}}\fint\limits_{B_{k}(z)\cap S^{j}}|f(z)-f(z')|^{p}\,d\mathcal{H}_{\theta_{j}}(z')\,d\mathcal{H}_{\theta_{j}}(z).
\end{split}
\end{equation}
As a result, combining estimates \eqref{eqq.4.17}, \eqref{eqq.4.18} and taking into account Theorem \ref{Th.equivalent_Besov_norms} we have
\begin{equation}
\label{eqq.4.19}
\sum\limits_{i \neq j}\sum\limits_{k=1}^{\infty}2^{k(p-\theta_{i}-\theta_{j})}(J^{i,j}_{k}(1)+J^{i,j}_{k}(2)) \le C \sum\limits_{i=1}^{N}\|f|\operatorname{B}^{1-\frac{\theta_{i}}{p}}_{p}(S^{i})\|^{p}.
\end{equation}

Finally, collecting estimates \eqref{eqq.4.15}, \eqref{eqq.4.16}, \eqref{eqq.4.19} and taking into account \eqref{eqq.gluing_functionals}
we arrive at the required estimate \eqref{eqq.4.11'} completing the proof.

\end{proof}

Now we establish the first keystone result of this subsection. We recall Remark \ref{Rem.measure_plus_porosity}.

\begin{Th}
\label{Th.1}
For each $\sigma \in (0,\sigma(S)]$, there is a constant $C > 0$ such that
\begin{equation}
\label{eqq.4.20''}
\operatorname{BN}_{p,\{\mathfrak{m}_{k}\},\sigma}(f) \le C \Bigl(\sum\limits_{i=1}^{N}\|f|\operatorname{B}_{p}^{1-\frac{\theta_{i}}{p}}(S^{i})\|+\mathcal{GL}^{(1)}_{p}(f)\Bigr) \quad \text{for all} \quad
f \in \cap_{i=1}^{N}L_{p}(\mathcal{H}_{\theta_{i}}\lfloor_{S^{i}}).
\end{equation}
\end{Th}

\begin{proof}
Since the set $S$ is $\sigma$-porous and $\mu(S)=0$, by \eqref{eq.main3} we have
\begin{equation}
\Bigl(\mathcal{BN}_{p,\{\mathfrak{m}_{k}\},\sigma}(f)\Bigr)^{p}
=\sum\limits_{k=1}^{\infty} 2^{k(\theta-p)}\int\limits_{S}\Bigl(\mathcal{E}_{\mathfrak{m}_{k}}(f,B_{k}(x))\Bigr)^{p}\,d\mathfrak{m}_{k}(x).
\end{equation}
It will be convenient to split the rest of the proof into several steps.

\textit{Step 1.} Given $i,j \in \{1,...,N\}$ with $i \neq j$, for each $k \in \mathbb{N}_{0}$, we have $B_{k}(x) \cap S^{j} = \emptyset$ for all $x \in S^{i} \setminus S^{i,j}_{k}$.
Consequently, by \eqref{eqq.concrete_measure}, we get
\begin{equation}
\notag
\int\limits_{S^{i} \setminus S^{i,j}_{k}}\Bigl(\mathcal{E}_{\mathfrak{m}_{k}}(f,B_{k}(x))\Bigr)^{p}\,d\mathfrak{m}_{k}(x)=2^{k(\theta-\theta_{i})}\int\limits_{S^{i} \setminus S^{i,j}_{k}}\Bigl(\mathcal{E}_{\mathcal{H}_{\theta_{i}}\lfloor_{S^{i}}}(f,B_{k}(x))\Bigr)^{p}\,d\mathcal{H}_{\theta_{i}}(x).
\end{equation}
Hence, taking into account \eqref{eqq.true_Besov_norm} we obtain
\begin{equation}
\label{eqq.4.22'}
\sum\limits_{i \neq j}\sum\limits_{k=1}^{\infty}2^{k(p-\theta)}\int\limits_{S^{i} \setminus S^{i,j}_{k}}\Bigl(\mathcal{E}_{\mathfrak{m}_{k}}(f,B_{k}(x))\Bigr)^{p}\,d\mathfrak{m}_{k}(x) \le \sum\limits_{i=1}^{N}\|f|\operatorname{B}_{p}^{1-\frac{\theta_{i}}{p}}(S^{i})\|^{p}.
\end{equation}

\textit{Step 2.} We fix for a moment $i,j \in \{1,...,N\}$ with $i \neq j$.
We recall notation given right after the proof of Proposition \ref{Prop.comparison_of_measures}.
Given $k \in \mathbb{N}_{0}$ and $\alpha \in \mathcal{A}_{k}(S)$ with $B_{k,\alpha} \cap S^{i,j}_{k} \neq \emptyset$,
we apply Proposition \ref{Prop.combinatorial_trick}. This gives an index set $\mathcal{I}_{k,\alpha} \subset \{1,...,N\}$ and a constant $c_{k,\alpha} \in \{1,...,N+1\}$ such that
$c_{k,\alpha}B_{k,\alpha} \cap S^{j} = \emptyset$ for all $j \in \{1,...,N\} \setminus \mathcal{I}_{k,\alpha}$ and for each $i \in \mathcal{I}_{k,\alpha}$ there is $x_{k,\alpha}(i) \in S^{i}$
such that $B_{k}(x_{k,\alpha}(i)) \subset c_{k,\alpha}B_{k,\alpha}$. Furthermore, without loss of generality we may
assume that $B_{k}(x) \subset c_{k,\alpha}B_{k,\alpha}$ for all $x \in S^{i,j}_{k} \cap B_{k,\alpha}$.
As a result, by Proposition \ref{Prop.oscillation_comparison}, \eqref{eqq.averaging_comparison} in combination with Lemma \ref{Lm.estimate_double_averaging_2} and H\"older's inequality we get
\begin{equation}
\label{eqq.4.9}
\begin{split}
&\int\limits_{S^{i,j}_{k} \cap B_{k,\alpha}}\Bigl(\mathcal{E}_{\mathfrak{m}_{k}}(f,B_{k}(x))\Bigr)^{p}\,d\mathfrak{m}_{k}(x) \le C 2^{k\theta}\mu(B_{k,\alpha})
\Bigl(\mathcal{E}_{\mathfrak{m}_{k}}(f,c_{k,\alpha}B_{k,\alpha})\Bigr)^{p}\\
&\le C 2^{k\theta}\mu(B_{k,\alpha})\Bigl[\sum\limits_{i' \in \mathcal{I}_{k,\alpha}}\Bigl(\mathcal{E}_{\mathcal{H}_{\theta_{i'}}\lfloor_{S^{i'}}}(f,c_{k,\alpha}B_{k,\alpha} \cap S^{i'})\Bigr)^{p}\\
&+\sum\limits_{\substack{i',j' \in \mathcal{I}_{k,\alpha} \\ i' \neq j'}}\fint\limits_{c_{k,\alpha}B_{k,\alpha}\cap S^{i'}}\fint\limits_{c_{k,\alpha}B_{k,\alpha} \cap S^{j'}}|f(y')-f(z')|^{p}\,d\mathcal{H}_{\theta_{i'}}(y')d\mathcal{H}_{\theta_{j'}}(z')\Bigr].
\end{split}
\end{equation}

\textit{Step 3.} Clearly, given $i' \in \mathcal{I}_{k,\alpha}$,
$c_{k,\alpha}B_{k,\alpha} \subset B_{k-N-1}(x)$ for all $x \in c_{k,\alpha}B_{k,\alpha} \cap S^{i'}$. Hence, by \eqref{eqq.doubling} and Definition \ref{Def.Ahlfors_David}
(recall that $B_{k}(x) \subset c_{k,\alpha}B_{k,\alpha}$ for all $x \in B_{k,\alpha} \cap S^{i'}$),
\begin{equation}
\begin{split}
\notag
&2^{k\theta}\mu(B_{k,\alpha}) \le 2^{k\theta}\mu(B_{k-N-1}(x)) \le C 2^{k\theta}\mu(B_{k}(x)) \\
&\le C 2^{k(\theta-\theta_{i'})}\mathcal{H}_{\theta_{i'}}(B_{k}(x) \cap S^{i'}) \le C 2^{k(\theta-\theta_{i'})}\mathcal{H}_{\theta_{i'}}(c_{k,\alpha}B_{k,\alpha} \cap S^{i'}).
\end{split}
\end{equation}
Combining the above observations with Proposition \ref{Prop.oscillation_comparison} (we apply this proposition with $\mathfrak{m}_{k}$ replaced by $2^{k(\theta-\theta_{i'})}\mathcal{H}_{\theta_{i'}}\lfloor_{S^{i'}}$), we derive, for each index $i' \in \mathcal{I}_{k,\alpha}$,
\begin{equation}
\label{eqq.4.24'}
\begin{split}
&2^{k\theta}\mu(B_{k,\alpha})\Bigl(\mathcal{E}_{\mathcal{H}_{\theta_{i'}}\lfloor_{S^{i'}}}(f,c_{k,\alpha}B_{k,\alpha} \cap S^{i'})\Bigr)^{p}\\
&\le C 2^{k(\theta-\theta_{i'})}\int\limits_{c_{k,\alpha}B_{k,\alpha} \cap S^{i'}}
\Bigl(\mathcal{E}_{\mathcal{H}_{\theta_{i'}}\lfloor_{S^{i'}}}(f,B_{k-N-1}(x) \cap S^{i'})\Bigr)^{p}\,d\mathcal{H}_{\theta_{i'}}(x).
\end{split}
\end{equation}

By Proposition \ref{Prop.weight_ineq} and \eqref{eqq.weights} we see that, given $i',j' \in \mathcal{I}_{k,\alpha}$, the following inequality
$$
\frac{1}{\mu(B_{k,\alpha})} \le C\operatorname{w}_{k-N-2}(y,z)
$$
holds for any $y \in c_{k,\alpha}B_{k,\alpha} \cap S^{i'}$
and any $z \in c_{k,\alpha}B_{k,\alpha} \cap S^{j'}$. Hence,
\begin{equation}
\label{eqq.4.26'}
\begin{split}
&2^{k\theta}\mu(B_{k,\alpha})\fint\limits_{c_{k,\alpha}B_{k,\alpha}\cap S^{i'}}\fint\limits_{c_{k,\alpha}B_{k,\alpha} \cap S^{j'}}|f(y)-f(z)|^{p}\,d\mathcal{H}_{\theta_{i'}}(y)d\mathcal{H}_{\theta_{j'}}(z)\\
&\le C 2^{k(\theta-\theta_{i'}-\theta_{j'})}\int\limits_{c_{k,\alpha}B_{k,\alpha}\cap S^{i'}}
\int\limits_{c_{k,\alpha}B_{k,\alpha} \cap S^{j'}}\operatorname{w}_{k-N-2}(y,z)|f(y)-f(z)|^{p}\,d\mathcal{H}_{\theta_{i'}}(y)d\mathcal{H}_{\theta_{j'}}(z).
\end{split}
\end{equation}

\textit{Step 4.}
If $k \geq N+3$,  $\alpha \in \mathcal{A}_{k}(S)$ and $i',j' \in \mathcal{I}_{k,\alpha}$ are such that $B_{k,\alpha} \cap S^{i,j}_{k} \neq \emptyset$, then by \eqref{eqq.sigma_set} we clearly have
\begin{equation}
\notag
c_{k,\alpha}B_{k,\alpha}\cap S^{i'} \times c_{k,\alpha}B_{k,\alpha}\cap S^{j'} \subset \Sigma^{i',j'}_{k-N-2}.
\end{equation}
Keeping in mind this observation, we combine \eqref{eqq.4.9}, \eqref{eqq.4.24'}, \eqref{eqq.4.26'} and take into account
Proposition \ref{Prop.overlapping}. For each $k \in \mathbb{N}_{0}$, $k \geq N+3$ we have
\begin{equation}
\notag
\begin{split}
&\int\limits_{S^{i,j}_{k}}\Bigl(\mathcal{E}_{\mathfrak{m}_{k}}(f,B_{k}(x))\Bigr)^{p}\,d\mathfrak{m}_{k}(x)=\sum\limits_{\alpha \in \mathcal{A}_{k}(S)}
\int\limits_{S^{i,j}_{k} \cap B_{k,\alpha}}\Bigl(\mathcal{E}_{\mathfrak{m}_{k}}(f,B_{k}(x))\Bigr)^{p}\,d\mathfrak{m}_{k}(x)\\
&\le  C\sum\limits_{i=1}^{N}2^{k(\theta-\theta_{i})}\int\limits_{S^{i}}
\Bigl(\mathcal{E}_{\mathcal{H}_{\theta_{i}}\lfloor_{S^{i}}}(f,B_{k-N-1}(y))\Bigr)^{p}\,d\mathcal{H}_{\theta_{i}}(y)\\
&+C\sum\limits_{i \neq j}2^{k(\theta-\theta_{i}-\theta_{j})}\iint\limits_{\Sigma^{i,j}_{k-N-2}}\operatorname{w}_{k-N-2}(y,z)|f(y)-f(z)|^{p}\,d\mathcal{H}_{\theta_{i}}(y)d\mathcal{H}_{\theta_{j}}(z).
\end{split}
\end{equation}
As a result, using \eqref{eqq.true_Besov_norm} and \eqref{eqq.gluing_functionals}, we obtain
\begin{equation}
\label{eqq.4.27'}
\begin{split}
&\sum\limits_{i \neq j}\sum\limits_{k=N+3}^{\infty}2^{k(p-\theta)}\int\limits_{S^{i,j}_{k}}\Bigl(\mathcal{E}_{\mathfrak{m}_{k}}(f,B_{k}(x))\Bigr)^{p}\,d\mathfrak{m}_{k}(x) \le
C\Bigl(\sum\limits_{i=1}^{N}\|f|\operatorname{B}_{p}^{1-\frac{\theta_{i}}{p}}(S^{i})\|^{p}+(\mathcal{GL}^{(1)}_{p}(f))^{p}\Bigr).
\end{split}
\end{equation}

\textit{Step 5.} It remains to note that by Proposition \ref{Prop.different_lp_norms} and Lemma \ref{Lm.estimate_by_lp_norms}
\begin{equation}
\label{eqq.4.28'}
\sum\limits_{k=1}^{N+2} 2^{k(\theta-p)}\int\limits_{S}\Bigl(\mathcal{E}_{\mathfrak{m}_{k}}(f,B_{k}(x))\Bigr)^{p}\,d\mathfrak{m}_{k}(x)
\le C \sum\limits_{i=1}^{N}\|f|L_{p}(\mathcal{H}_{\theta_{i}}\lfloor_{S^{i}})\|^{p}.
\end{equation}

\textit{Step 6.} Combining \eqref{eqq.4.22'}, \eqref{eqq.4.27'} and \eqref{eqq.4.28'}, we arrive at \eqref{eqq.4.20''}, completing the proof.
\end{proof}

\begin{Th}
\label{Th.2}
For each $\sigma \in (0,\sigma(S)]$, there is a constant $C > 0$ such that
\begin{equation}
\label{eqq.4.20'}
\Bigl(\sum\limits_{i=1}^{N}\|f|\operatorname{B}_{p}^{1-\frac{\theta_{i}}{p}}(S^{i})\|+\mathcal{GL}^{(3)}_{p}(f)\Bigr)  \le C \operatorname{BN}_{p,\{\mathfrak{m}_{k}\},\sigma}(f) \quad \text{for all} \quad
f \in L_{p}(\mathfrak{m}_{0}).
\end{equation}
\end{Th}

\begin{proof}
Given $i,j \in \{1,...,N\}$, for each $k \in \mathbb{N}_{0}$,
we combine Lemma \ref{Lm.estimate_double_averaging_1} (applied with $c=2$ and $B=B_{k-1}(y)$) with Proposition \ref{Prop.integral_of_weights}. This yields
\begin{equation}
\notag
\begin{split}
&\iint\limits_{\Sigma^{i,j}_{k}}\operatorname{w}_{k}(y,z)(A^{i,j}_{k}(f)(y,z))^{p}\,d\mathcal{H}_{\theta_{i}}(y)d\mathcal{H}_{\theta_{j}}(z) \le C 2^{k\theta_{j}}
\int\limits_{S^{i,j}_{k}}\Bigl(\mathcal{E}_{\mathfrak{m}_{k}}(f,B_{k-1}(y))\Bigr)^{p}\,d\mathcal{H}_{\theta_{i}}(y)\\
&\le C 2^{k(\theta_{i}+\theta_{j}-\theta)}\int\limits_{S^{i}}\Bigl(\mathcal{E}_{\mathfrak{m}_{k}}(f,B_{k-1}(y))\Bigr)^{p}\,d\mathfrak{m}_{k}(y).
\end{split}
\end{equation}
Hence, using (\textbf{M}4) in Definition \ref{Def.regular_sequence} and Lemma \ref{Lm.estimate_by_lp_norms} (with $L=0$), we obtain
\begin{equation}
\begin{split}
&\Bigl(\mathcal{GL}^{(3)}_{p}(f)\Bigr)^{p} \le C \sum\limits_{k=1}^{\infty}2^{k(p-\theta)}\int\limits_{S}\Bigl(\mathcal{E}_{\mathfrak{m}_{k-1}}(f,B_{k-1}(y))\Bigr)^{p}\,d\mathfrak{m}_{k-1}(y)\\
&\le
C \Bigl(\operatorname{BN}_{p,\{\mathfrak{m}_{k}\},\sigma}(f)\Bigr)^{p}.
\end{split}
\end{equation}
Finally, combining the above inequality with Proposition \ref{Prop.estimate_simplebesov_by_modifiedbesov} we complete the proof.
\end{proof}

Combining Theorems \ref{Th.main}, \ref{Th.1}, \ref{Th.2} with Lemmas \ref{Lm.1}, \ref{Lm.2} and Remark \ref{Rem.gluing_ineq_1} we immediately obtain
\textit{the main result} of this subsection.

\begin{Ca}
\label{Ca.1}
A function $f \in \cap_{i=1}^{N}L_{p}(\mathcal{H}_{\theta_{i}}\lfloor_{S^{i}})$ belongs to the space $W_{p}^{1}|_{S}^{\mathfrak{m}_{0}}$ if and only if $f \in \cap_{i=1}^{N}\operatorname{B}^{1-\frac{\theta_{i}}{p}}_{p}(\mathcal{H}_{\theta_{i}}\lfloor_{S^{i}})$ and $\mathcal{GL}^{(l)}_{p}(f) < +\infty$ for some $l \in \{1,2,3\}$.
Furthermore,
\begin{equation}
\|f|W_{p}^{1}(\operatorname{X})|_{S}^{\mathfrak{m}_{0}}\| \approx \sum\limits_{i=1}^{N}\|f|\operatorname{B}_{p}^{1-\frac{\theta_{i}}{p}}(S^{i})\|+\mathcal{GL}^{(l)}_{p}(f), \quad l \in \{1,2,3\},
\end{equation}
where the equivalence constants do not depend on $f$.

Finally, there exists an $\mathfrak{m}_{0}$-extension operator $\operatorname{Ext}_{S,\{\mathfrak{m}_{k}\}} \in \mathcal{L}(W_{p}^{1}(\operatorname{X})|_{S}^{\mathfrak{m}_{0}},W_{p}^{1}(\operatorname{X}))$.
\end{Ca}

\subsection{Difficult case} In this subsection, we consider a more complicated case when $\theta_{1}=0$.
For the technical simplicity we assume that $N=2$, and, hence, $\theta_{2} > 0$. In contrast with the previous subsection the resulting
trace space will be a mixture of \textit{``spaces of different nature''.} Roughly speaking, 
the trace norm will be composed of the Sobolev-type seminorm, the Besov-type norm and
the corresponding gluing functional.
Apart from the ideological difference with the previous subsection, in this case we should overcome a \textit{technical difficulty}.
More precisely, since $\theta_{1}=0$, the set $S = S^{1} \cup S^{2}$ is not necessary porous.

Under the above assumptions, we clearly have
\begin{equation}
\label{eqq.special_sequence_2}
\mathfrak{m}_{k}=2^{k\theta}\mu\lfloor_{S^{1}}+2^{k(\theta-\theta_{2})}\mathcal{H}_{\theta_{2}}\lfloor_{S^{2}}, \quad k \in \mathbb{N}_{0}.
\end{equation}
Keeping in mind Definition \ref{Def.Ahlfors_David} and \eqref{eqq.tricky_average}, we put
\begin{equation}
\label{eqq.simplified_sharp}
f^{\sharp}_{\mu\lfloor_{S^{1}}}(x):=\sup\limits_{r \in (0,2]}\mathcal{E}_{\mu\lfloor_{S^{1}}}(f,B_{r}(x)), \quad x \in S.
\end{equation}

We recall the following lemma from \cite{T5} (we use notation $k(B):=k(r_{B})$).
\begin{Lm}
\label{Lm.largescale_estimate}
Let $\delta \in (0,1]$ and $c \geq 1$. Then there is a constant $C > 0$ depending on $\delta$, such that if $\mathcal{B}_{\delta}$ is an arbitrary $(S,c)$-nice family of balls such that $r(B) \geq \delta$
for all $B \in \mathcal{B}_{\delta}$, then, for each $f \in L_{p}(\mathfrak{m}_{0})$,
\begin{equation}
\sum\limits_{B \in \mathcal{B}_{\delta}} \frac{\mu(B)}{(r_{B})^{p}}\Bigl(\mathcal{E}_{\mathfrak{m}_{k(B)}}(f,2cB)\Bigr)^{p}
\le C \int\limits_{S}|f(x)|^{p}\,d\mathfrak{m}_{0}(x).
\end{equation}
\end{Lm}

We also recall a combinatorial result, which is a slight modification of Theorem 2.6 in \cite{Shv1}.
\begin{Prop}
\label{Prop.10.1}
Let $c \geq 1$ and let $\mathcal{B}$ be an $(S,c)$-Whitney family of balls.
Then there exist constants $C > 0$, $\tau \in (0,1)$, and a family $\mathcal{U}:=\{U(B):B \in \mathcal{B}\}$ of Borel subsets of $S$ such that
$U(B) \subset 2cB$, $\mu(U(B)) \geq \tau \mu(B)$ for all $B \in \mathcal{B}$, and the covering multiplicity of the family $\{U(B):B \in \mathcal{B}\}$ is bounded above by $C$.
\end{Prop}

The first useful technical observation is given by the following lemma.

\begin{Lm}
\label{Lm.3}
For each $c \geq 1$, there is a constant $C > 0$ such that if $\mathcal{F}:=\{B_{r_{i}}(x_{i})\}_{i=1}^{\overline{N}}$, $\overline{N} \in \mathbb{N}$, is an $(S^{1},c)$-nice family  with
$\max\{4cr(B):B \in \mathcal{F}\} \le 1$, then
\begin{equation}
\label{eqq.4.36}
\sum\limits_{i=1}^{\overline{N}} \frac{\mu(B_{r_{i}}(x_{i}))}{r^{p}_{i}}\Bigl(\mathcal{E}_{\mu\lfloor_{S^{1}}}(f,B_{2cr_{i}}(x_{i}))\Bigr)^{p}
\le C\int\limits_{S^{1}}(f^{\sharp}_{\mu\lfloor_{S^{1}}})^{p}\,d\mu(x).
\end{equation}
\end{Lm}

\begin{proof}
Consider the family $\mathcal{F}_{1}:=\{B \in \mathcal{F}:\frac{1}{2}B \cap S^{1} \neq \emptyset\}$.
Given a ball $B=B_{r}(\underline{x}) \in \mathcal{F}_{1}$ we fix a point $x_{B} \in \frac{1}{2}B \cap S^{1}$. Clearly, we have the following inclusions
\begin{equation}
\notag
B_{\frac{r}{2}}(x_{B}) \subset B \subset 2cB \subset B_{(2c+1)r}(x), \quad x \in B \cap S^{1}.
\end{equation}
Using the above inclusions, \eqref{eqq.doubling} and \eqref{eqq.Ahlfors_David} (for $\theta=0$) we get
\begin{equation}
\notag
\mu(B \cap S) \le \mu(B) \le C\mu(B_{\frac{r}{2}}(x_{B})) \le C\mu(B_{\frac{r}{2}}(x_{B}) \cap S) \le C \mu(B \cap S).
\end{equation}
Hence, applying Proposition \ref{Prop.oscillation_comparison} with $\mathfrak{m}_{k}=2^{k\theta}\mu\lfloor_{S^{1}}$, $k \in \mathbb{N}_{0}$, we obtain
\begin{equation}
\notag
\mu(B)\Bigl(\mathcal{E}_{\mu\lfloor_{S^{1}}}(f,2cB)\Bigr)^{p} \le C\mu(B \cap S)\Bigl(\mathcal{E}_{\mu\lfloor_{S^{1}}}(f,(2c+1)B_{r}(x))\Bigr)^{p} \quad \text{for all} \quad x \in B \cap S^{1}.
\end{equation}
As a result, since the family $\mathcal{F}_{1}$ is disjoint, we get (we take into account \eqref{eqq.simplified_sharp})
\begin{equation}
\label{eqq.4.38}
\sum\limits_{B \in \mathcal{F}_{1}}\frac{\mu(B)}{(r(B))^{p}}\Bigl(\mathcal{E}_{\mu\lfloor_{S^{1}}}(f,2cB)\Bigr)^{p}
\le C\sum\limits_{B \in \mathcal{F}_{1}}\int\limits_{B \cap S^{1}}(f^{\sharp}_{\mu\lfloor_{S^{1}}}(x))^{p}\,d\mu(x)
\le C\int\limits_{S^{1}}(f^{\sharp}_{\mu\lfloor_{S^{1}}})^{p}\,d\mu(x).
\end{equation}

Consider the family $\mathcal{F}_{2}:=\{\frac{1}{2}B: B \in \mathcal{F} \setminus \mathcal{F}_{1}\}$. Clearly, $\mathcal{F}_{2}$ is an $(S^{1},2c)$-Whitney family of balls.
Using Proposition \ref{Prop.oscillation_comparison} with $\mathfrak{m}_{k}=2^{k\theta}\mu\lfloor_{S^{1}}$ and taking into account Proposition \ref{Prop.10.1} we get, for each $B \in \mathcal{F}_{2}$,
\begin{equation}
\notag
\mu(B)\Bigl(\mathcal{E}_{\mu\lfloor_{S^{1}}}(f,4cB)\Bigr)^{p} \le C\mu(U(B))\Bigl(\mathcal{E}_{\mu\lfloor_{S^{1}}}(f,(4c+1)B_{r(B)}(x))\Bigr)^{p} \quad \text{for all} \quad x \in U(B).
\end{equation}
It follows from Proposition \ref{Prop.10.1} that for some $C > 0$ we have
$$
\sup\limits_{x \in \operatorname{X}}\sum\limits_{B \in \mathcal{F}_{2}}\chi_{U(2B)}(x) \le C.
$$
As a result, we obtain (we take into account \eqref{eqq.simplified_sharp})
\begin{equation}
\label{eqq.4.39}
\sum\limits_{B \in \mathcal{F}_{2}} \frac{\mu(B)}{(r_{B})^{p}}\Bigl(\mathcal{E}_{\mu\lfloor_{S^{1}}}(f,4cB)\Bigr)^{p}
\le C\sum\limits_{B \in \mathcal{F}_{2}}\int\limits_{U(B)}(f^{\sharp}_{\mu\lfloor_{S^{1}}})^{p}\,d\mu(x)
\le C\int\limits_{S^{1}}(f^{\sharp}_{\mu\lfloor_{S^{1}}})^{p}\,d\mu(x).
\end{equation}
Combining \eqref{eqq.4.38} and \eqref{eqq.4.39} we obtain the required estimate and complete the proof.

\end{proof}

The second useful technical observation is given by the following lemma.

\begin{Lm}
\label{Lm.4}
Given $c \geq 1$, there is a constant $C > 0$ such that if $\mathcal{B}:=\{B_{r_{i}}(x_{i})\}_{i=1}^{\overline{N}}$, $\overline{N} \in \mathbb{N}$, is an $(S^{2},c)$-nice family  with
$\max\{8cr_{i}: 1 \le i \le \overline{N}\} \le 1$, then
\begin{equation}
\label{eqq.4.36}
\sum\limits_{i=1}^{\overline{N}} \frac{\mu(B_{r_{i}}(x_{i}))}{r^{p}_{i}}\Bigl(\mathcal{E}_{\mathcal{H}_{\theta_{2}}\lfloor_{S^{2}}}(f,B_{2cr_{i}}(x_{i}))\Bigr)^{p}
\le C\|f|\operatorname{B}_{p}^{1-\frac{\theta_{2}}{p}}(S^{2})\|^{p}.
\end{equation}
\end{Lm}

\begin{proof}
It is easy to see that, given a ball $B=B_{r_{i}}(x_{i}) \in \mathcal{B}$, there is a ball $\widetilde{B}$ of the same radius
centered at some point $\tilde{x}_{i} \in S^{2}$ such that $\widetilde{B} \subset 2cB$. Furthermore, it is clear that $2cB \subset B_{4cr_{i}}(x)$
for all $x \in 2cB \cap S^{2}$. Given $k \in \mathbb{N}_{0}$, we put $\mathcal{B}(k):=\{B \in \mathcal{B}:r_{B} \in (2^{-k-1},2^{-k}]\}$.
Hence, using Definition \ref{Def.Ahlfors_David} and applying Proposition \ref{Prop.oscillation_comparison} with $\mathfrak{m}_{k}=2^{k(\theta-\theta_{2})}\mathcal{H}_{\theta_{2}}\lfloor_{S^{2}}$,
it is easy to see that, for each ball $B \in \mathcal{B}(k)$,
\begin{equation}
\mu(B)\Bigl(\mathcal{E}_{\mathcal{H}_{\theta_{2}}\lfloor_{S^{2}}}(f,2cB)\Bigr)^{p} \le 2^{-k\theta_{2}}\mathcal{H}_{\theta_{2}}(2cB \cap S^{2})\inf\limits_{x \in 2cB \cap S^{2}}\Bigl(\mathcal{E}_{\mathcal{H}_{\theta_{2}}\lfloor_{S^{2}}}(f,4cB_{k}(x))\Bigr)^{p}.
\end{equation}
Hence, taking into account that the covering multiplicity of $\{2cB:B \in \mathcal{B}(k)\}$ is bounded above by some constant $C > 0$ independent on $k$, we see that
\begin{equation}
\label{eqq.3.1.77}
\begin{split}
&\sum\limits_{B \in \mathcal{B}}\frac{\mu(B)}{(r_{B})^{p}}\Bigl(\mathcal{E}_{\mathcal{H}_{\theta_{2}}\lfloor_{S^{2}}}(f,2cB)\Bigr)^{p} \\
&\le C\sum\limits_{k=0}^{\infty}2^{k(p-\theta_{2})}\sum\limits_{B \in \mathcal{B}(k)}\int\limits_{2cB \cap S^{2}}\Bigl(\mathcal{E}_{\mathcal{H}_{\theta_{2}}\lfloor_{S^{2}}}(f,4cB_{k}(x))\Bigr)^{p} \,d\mathcal{H}_{\theta_{2}}(x)\\
&\le
C \sum\limits_{k:\mathcal{B}(k) \neq \emptyset}2^{k(p-\theta_{2})}\int\limits_{S^{2}}\Bigl(\mathcal{E}_{\mathcal{H}_{\theta_{2}}\lfloor_{S^{2}}}(f,4cB_{k}(x))\Bigr)^{p}\,d\mathcal{H}_{\theta_{2}}(x).
\end{split}
\end{equation}
In accordance with the assumptions of the lemma $\max\{8cr_{B}:B \in \mathcal{B}\} \le 1$. Thus, if $B \in \mathcal{B}(k)$ for some $k \in \mathbb{N}_{0}$, then
$4c2^{-k} \le 1$. Hence, for each $k \in \mathbb{N}_{0}$ with $\mathcal{B}(k) \neq \emptyset$ we have $2^{-j(k)} \le 1$, where $j(k)$ is the maximum among all 
$j \in \mathbb{N}_{0}$ satisfying $4c2^{-k} \le 2^{-j}$. Taking into account this observation we continue \eqref{eqq.3.1.77} and get 
\begin{equation}
\notag
\sum\limits_{B \in \mathcal{B}}\frac{\mu(B)}{(r_{B})^{p}}\Bigl(\mathcal{E}_{\mathcal{H}_{\theta_{2}}\lfloor_{S^{2}}}(f,2cB)\Bigr)^{p}
\le C \sum\limits_{j=0}^{\infty}2^{j(p-\theta_{2})}\int\limits_{S^{2}}\Bigl(\mathcal{E}_{\mathcal{H}_{\theta_{2}}\lfloor_{S^{2}}}(f,B_{j}(x))\Bigr)^{p}\,d\mathcal{H}_{\theta_{2}}(x).
\end{equation}
As a result, taking into account Definition \ref{Def.Besov} we complete the proof.
\end{proof}

Now we are ready to establish the first crucial result in this subsection.

\begin{Th}
\label{Th.3}
For each $c \geq 1$, there exists a constant $C > 0$ such that
\begin{equation}
\label{eqq.4.44'}
\mathcal{BSN}_{p,\{\mathfrak{m}_{k}\},c}(f) \le C \Bigl(\|f^{\sharp}_{\mu\lfloor_{S^{1}}}|L_{p}(S^{1},\mu)\|
+\|f|L_{p}(S^{1},\mu)\|+\|f|\operatorname{B}_{p}^{1-\frac{\theta}{p}}(S^{2})\|+\mathcal{GL}^{(3)}_{p}(f)\Bigr).
\end{equation}
\end{Th}

\begin{proof}
Let $\underline{k} \in \mathbb{N}$ be the minimal among all $k \in \mathbb{N}$ satisfying $2^{-k} < 1/4c$. 
Let $\mathcal{B}$ be an arbitrary $(S,c)$-nice family of closed balls. 
We define the auxiliary families $\underline{\mathcal{B}}:=\{B \in \mathcal{B}:r_{B} \geq 2^{-\underline{k}-1}\}$ and $\overline{\mathcal{B}}:=\mathcal{B} \setminus \underline{B}$. 
We put
\begin{equation}
\mathcal{B}^{1}:=\{\overline{B} \in \mathcal{B}:cB \cap S^{2}=\emptyset\}, \quad \mathcal{B}^{2}:=\{B \in \overline{\mathcal{B}}:cB \cap S^{1}=\emptyset\}.
\end{equation}
For each $B \in \mathcal{B}$ we set $k(B):=k(r_{B})$, as usual. Recall \eqref{eqq.tricky_average}. We split the rest of the proof into several steps.

\textit{Step 1.}
For each $B \in \mathcal{B}^{1}$ we have $\widetilde{\mathcal{E}}_{\mathfrak{m}_{k(B)}}(f,cB)=\mathcal{E}_{\mu\lfloor_{S^{1}}}(f,2cB))$. Hence, by Lemma \ref{Lm.3}
\begin{equation}
\begin{split}
\label{eqq.4.46}
&\sum\limits_{B \in \mathcal{B}^{1}} \frac{\mu(B)}{(r_{B})^{p}}\Bigl(\widetilde{\mathcal{E}}_{\mathfrak{m}_{k(B)}}(f,cB)\Bigr)^{p} \le C\int\limits_{S^{1}}(f^{\sharp}_{\mu\lfloor_{S^{1}}})^{p}\,d\mu(x).
\end{split}
\end{equation}

\textit{Step 2.}
For each $B \in \mathcal{B}^{2}$ we have $\widetilde{\mathcal{E}}_{\mathfrak{m}_{k(B)}}(f,cB)=\mathcal{E}_{\mathcal{H}_{\theta_{2}}\lfloor_{S^{2}}}(f,2cB))$. By Lemma \ref{Lm.4} we get
\begin{equation}
\label{eqq.4.47}
\sum\limits_{B \in \mathcal{B}^{2}} \frac{\mu(B)}{(r_{B})^{p}}\Bigl(\widetilde{\mathcal{E}}_{\mathfrak{m}_{k(B)}}(f,cB)\Bigr)^{p}
\le C\|f|\operatorname{B}_{p}^{1-\frac{\theta_{2}}{p}}(S^{2})\|^{p}.
\end{equation}

\textit{Step 3.}
We put $\mathcal{B}^{3}:=\overline{\mathcal{B}} \setminus (\mathcal{B}^{1}\cup\mathcal{B}^{2})$, i.e.,
$B \in \mathcal{B}^{3}$ if and only if $cB \cap S^{1} \neq \emptyset$, $cB \cap S^{2} \neq \emptyset$
and $B \in \overline{\mathcal{B}}$.
Given $k \in \mathbb{Z}$, we consider the family
\begin{equation}
\label{eqq.ball_with_wave}
\mathcal{B}^{3}(k):=\{B \in \mathcal{B}^{3}:r(B) \in (2^{-k-1},2^{-k}]\}. 
\end{equation}

Note that, given $k \in \mathbb{N}_{0}$, for each ball $B \in \mathcal{B}^{3}(k)$, there exist points $x^{1}_{k}(B) \in S^{1}$ and $x^{2}_{k}(B) \in S^{2}$
such that
\begin{equation}
\label{eqq.imp.incl}
B_{k}(x^{i}_{k}(B)) \subset 2cB \subset 3cB_{k}(x^{i}_{k}(B)), \quad i=1,2.
\end{equation}
Thus, an application of Lemma \ref{Lm.estimate_double_averaging_2} gives
\begin{equation}
\label{eqq.4.49}
\begin{split}
&\mathcal{E}_{\mathfrak{m}_{k}}(f,2cB) \le C\Bigl(\mathcal{E}_{\mu\lfloor_{S^{1}}}(f,2cB)+\mathcal{E}_{\mathcal{H}_{\theta_{2}}\lfloor_{S^{2}}}(f,2cB)\\
&+\fint\limits_{2cB \cap S^{1}}\fint\limits_{2cB \cap S^{2}}|f(y')-f(z')|\,d\mu(y')d\mathcal{H}_{\theta_{2}}(z')\Bigr).
\end{split}
\end{equation}

\textit{Step 4.} From Lemmas \ref{Lm.3}, \ref{Lm.4} it follows that
\begin{equation}
\begin{split}
\label{eqq.4.50}
&\sum\limits_{B \in \mathcal{B}^{3}} \frac{\mu(B)}{(r_{B})^{p}}\Bigl[\Bigl(\mathcal{E}_{\mu\lfloor_{S^{1}}}(f,2cB)\Bigr)^{p}
+\Bigl(\mathcal{E}_{\mathcal{H}_{\theta_{2}}\lfloor_{S^{2}}}(f,2cB)\Bigr)^{p}\Bigr]\\
&\le C\Bigl[\int\limits_{S^{1}}(f^{\sharp}_{\mu\lfloor_{S^{1}}})^{p}\,d\mu(x)+\|f|\operatorname{B}_{p}^{1-\frac{\theta_{2}}{p}}(S^{2})\|^{p}\Bigr].
\end{split}
\end{equation}

\textit{Step 5.}
Fix for a moment $k \in \mathbb{N}_{0}$ and $B \in \mathcal{B}^{3}(k)$.
By \eqref{eqq.weight} and Proposition \ref{Prop.weight_ineq} it is easy to see that 
\begin{equation}
\notag
(\mu(2cB))^{-1} \le C\operatorname{w}_{k}(y,z) \quad
\text{for all} \quad (y,z) \in 2cB  \times 2cB.
\end{equation}
At the same time, combining Definition \ref{Def.Ahlfors_David} with inclusions \eqref{eqq.imp.incl} and taking into account \eqref{eqq.doubling} and Remark \ref{Rem.Ahlfors_doubling}
we get
\begin{equation}
\notag
\mu(2cB) \le C\mu(2cB \cap S^{1}), \quad \mu(B)2^{k\theta_{2}} \le C \mathcal{H}_{\theta_{2}}(2cB \cap S^{2}).
\end{equation}
As a result, we deduce existence of $C > 0$ such that, for each $(y,z) \in 2cB \cap S^{1} \times 2cB \cap S^{2}$,
\begin{equation}
\begin{split}
\label{eqq.useful_measure_estimate}
&\frac{\mu(B)}{(r_{B})^{p}} \le C 2^{kp}\mu(B) \frac{\mathcal{H}_{\theta_{2}}(2cB \cap S^{2})}{\mathcal{H}_{\theta_{2}}(2cB \cap S^{2})}\mu(2cB)\operatorname{w}_{k}(y,z)\\
&\le C 2^{k(p-\theta_{2})}\mathcal{H}_{\theta_{2}}(2cB \cap S^{2})\mu(2cB \cap S^{1})\operatorname{w}_{k}(y,z).
\end{split}
\end{equation}
Furthermore, using \eqref{eqq.doubling}, Remark \ref{Rem.Ahlfors_doubling} and \eqref{eqq.ball_with_wave} it is easy to see that
\begin{equation}
\label{eqq.useful_average_estimate}
\fint\limits_{2cB \cap S^{1}}\fint\limits_{2cB \cap S^{2}}|f(y')-f(z')|\,d\mu(y')d\mathcal{H}_{\theta_{2}}(z') \le C \inf A^{1,2}_{k-\underline{k}}(f)(y,z),
\end{equation}
where the infimum is taken over all $(y,z) \in 2cB \cap S^{1} \times 2cB \cap S^{2}$.

\textit{Step 6.}
Given $k \in \mathbb{N}_{0}$, by Definition the covering multiplicity of the family $\mathcal{B}^{3}(k)$ is bounded above by some constant $C > 0$ independent of $k$.
Using this fact in combination with \eqref{eqq.useful_measure_estimate}, \eqref{eqq.useful_average_estimate}
and taking into account that $2cB \cap S^{1} \times 2cB \cap S^{2} \subset \Sigma^{1,2}_{k-\underline{k}}$ we obtain, for each $B \in \mathcal{B}^{3}(k)$, the following estimate
\begin{equation}
\label{eqq.4.52'}
\begin{split}
&\sum\limits_{k=\underline{k}+1}^{\infty}\sum\limits_{B \in \mathcal{B}^{3}(k)}\frac{\mu(B)}{(r_{B})^{p}}\Bigl(\fint\limits_{2cB \cap S^{1}}\fint\limits_{2cB \cap S^{2}}|f(y)-f(z)|\,d\mu(y)d\mathcal{H}_{\theta_{2}}(z)\Bigr)^{p}\\
&\le C \sum\limits_{k=\underline{k}+1}^{\infty}2^{k(p-\theta_{2})}\iint\limits_{\Sigma^{1,2}_{k-\underline{k}}}\operatorname{w}_{k-\underline{k}}(y,z)(A^{1,2}_{k-\underline{k}}(f)(y,z))^{p}\,d\mu(y)d\mathcal{H}_{\theta_{2}}(z) \le C \Bigl(\mathcal{GL}^{1}_{p}(f)\Bigr)^{p}.
\end{split}
\end{equation}
Finally, using Lemma \ref{Lm.largescale_estimate} in combination with Proposition \ref{Prop.different_lp_norms} we have
\begin{equation}
\label{eqq.4.53'}
\begin{split}
&\sum\limits_{B \in \underline{\mathcal{B}}}\frac{\mu(B)}{(r_{B})^{p}}\Bigl(\mathcal{E}_{\mathfrak{m}_{k}}(f,2cB)\Bigr)^{p} \le \|f|L_{p}(\mathfrak{m}_{0})\|^{p}.
\end{split}
\end{equation}

\textit{Step 7.}
Finally, combining estimates \eqref{eqq.4.46}, \eqref{eqq.4.47}, \eqref{eqq.4.49}, \eqref{eqq.4.50}, \eqref{eqq.4.52'}, \eqref{eqq.4.53'} we arrive at \eqref{eqq.4.44'} and complete the proof.
\end{proof}

In order to establish the following result we recall that the set $S^{2}$ is $\sigma_{2}(S)$-porous.

\begin{Th}
\label{Th.4}
There is a constant $\underline{c} \geq 1$ depending on $\sigma_{2}(S)$ such that the following holds. For each $c \geq \underline{c}$, there is a constant $C > 0$ such that
\begin{equation}
\mathcal{GL}^{(3)}_{p}(f) \le C \mathcal{BSN}_{p,\{\mathfrak{m}_{k}\},c}(f).
\end{equation}
\end{Th}

\begin{proof}
We split the proof into several steps. Given $k \in \mathbb{N}_{0}$, let $Z_{k}(S^{2})$ be an arbitrary 
maximal $2^{-k}$-separated subset of $S^{2}$ with the corresponding index set $\mathcal{A}_{k}(S^{2})$, i.e., 
$$
Z_{k}(S^{2})=\{z_{k,\alpha}:\alpha \in \mathcal{A}_{k}(S^{2})\}.
$$

\textit{Step 1.}
Arguing as in the proof of Theorem \ref{Th.2} we have
\begin{equation}
\label{eqq.4.55}
\begin{split}
\Bigl(\mathcal{GL}^{(3)}_{p}(f)\Bigr)^{p} \le C \sum\limits_{k=0}^{\infty}2^{k(p-\theta)}\int\limits_{S^{2}}\Bigl(\mathcal{E}_{\mathfrak{m}_{k}}(f,B_{k}(y))\Bigr)^{p}\,d\mathfrak{m}_{k}(y).
\end{split}
\end{equation}

\textit{Step 2.}
Since $S^{2}$ is $\sigma_{2}:=\sigma_{2}(S)$-porous, given $k \in \mathbb{N}_{0}$ and $\alpha \in \mathcal{A}_{k}(S^{2})$, there is a ball $\widehat{B}_{k,\alpha} \subset B_{k,\alpha} \setminus S^{2}$ with $r(\widehat{B}_{k,\alpha}) \geq \sigma_{2}r(B_{k,\alpha})$.
Hence, $\frac{3}{\sigma_{2}}\widehat{B}_{k,\alpha} \supset B_{k}(y)$ for all $y \in B_{k,\alpha} \cap S^{2}$.
This inclusion in combination with Proposition \ref{Prop.oscillation_comparison} implies
\begin{equation}
\label{eqq.4.56}
\mathcal{E}_{\mathfrak{m}_{k}}(f,B_{k}(y)) \le C \mathcal{E}_{\mathfrak{m}_{k}}(\frac{3}{\sigma_{2}}\widehat{B}_{k,\alpha},f) \quad \text{for all} \quad y \in B_{k,\alpha} \cap S^{2}.
\end{equation}

\textit{Step 3.}
By \eqref{eqq.doubling} and \eqref{eqq.4.56} it is clear that
\begin{equation}
\label{eqq.4.57}
\begin{split}
&\int\limits_{S^{2}}\Bigl(\mathcal{E}_{\mathfrak{m}_{k}}(f,B_{k}(y))\Bigr)^{p}\,d\mathfrak{m}_{k}(y) \le \sum\limits_{\alpha \in \mathcal{A}_{k}(S^{2})}\int\limits_{B_{k,\alpha} \cap S^{2}}\Bigl(\mathcal{E}_{\mathfrak{m}_{k}}(f,B_{k}(y))\Bigr)^{p}\,d\mathfrak{m}_{k}(y)\\
&\le C \sum\limits_{\alpha \in \mathcal{A}_{k}(S^{2})} 2^{k\theta}\mu(\widehat{B}_{k,\alpha})\mathcal{E}_{\mathfrak{m}_{k}}(\frac{3}{\sigma_{2}}\widehat{B}_{k,\alpha},f).
\end{split}
\end{equation}

\textit{Step 4.}
By Lemma 7.3 in \cite{T5} there is a constant $N_{1} \in \mathbb{N}_{0}$ such that, for each $k \in \mathbb{N}$, the family
$\{\widehat{B}_{k,\alpha}:\alpha \in \mathcal{A}_{k}(S^{2})\}$ can be decomposed into at most $N_{1}$ disjoint subfamilies.
Furthermore, since $\operatorname{dist}(\frac{1}{2}\widehat{B}_{k,\alpha},S^{2}) \geq \frac{\sigma_{2}}{2}2^{-k}$ there is a constant $N_{2} \in \mathbb{N}$ depending on $\sigma_{2}$ only
such that, for each $k \in \mathbb{N}_{0}$ and $\alpha \in \mathcal{A}_{k}(S^{2})$, every ball $\frac{1}{2}\widehat{B}_{k,\alpha}$ with $\alpha \in \mathcal{A}_{k}(S^{2})$ does not meet
any ball $\frac{1}{2}\widehat{B}_{k+N_{2},\beta}$, $\beta \in \mathcal{A}_{k+N_{2}}(S^{2})$.
As a result, there is $i \in \{1,...,N_{2}\}$ and a sequence of index sets $\{\mathcal{J}_{k}\}_{k=0}^{\infty}$ such that $\mathcal{J}_{k} \subset \mathcal{A}_{i+kN_{2}}$, $k \in \mathbb{N}$
and
\begin{equation}
\begin{split}
\label{eqq.4.58}
&\sum\limits_{k=0}^{\infty}2^{k(p-\theta)} \sum\limits_{\alpha \in \mathcal{A}_{k}(S^{2})} 2^{k\theta}\mu(\widehat{B}_{k,\alpha})\mathcal{E}_{\mathfrak{m}_{k}}(\frac{3}{\sigma_{2}}\widehat{B}_{k,\alpha},f)\\
&\le N_{1}N_{2}\sum\limits_{k=0}^{\infty}2^{(i+kN_{2})p}
\sum\limits_{\alpha \in \mathcal{J}_{k}} \mu(\widehat{B}_{i+kN_{2},\alpha})\mathcal{E}_{\mathfrak{m}_{i+kN_{2}}}(\frac{3}{\sigma_{2}}\widehat{B}_{i+kN_{2},\alpha},f).
\end{split}
\end{equation}

\textit{Step 5.} Taking into account that the family $\mathcal{B}:=\{B_{i+kN_{2},\alpha}: k \in \mathbb{N}_{0}, \alpha \in \mathcal{J}_{k}\}$ is disjoint, combining \eqref{eqq.4.55}, \eqref{eqq.4.57}, \eqref{eqq.4.58}
and letting $\underline{c}=\frac{3}{\sigma_{2}}$ we arrive at the required estimate.

The proof is complete.

\end{proof}

\textit{The main result} of this subsection is as follows.
\begin{Ca}
\label{Ca.2}
Let $\sigma \in (0,\sigma_{2}(S)]$. A function $f \in \cap_{i=1}^{2}L_{p}(\mathcal{H}_{\theta_{i}}\lfloor_{S^{i}})$ belongs to the space $W_{p}^{1}|_{S}^{\mathfrak{m}_{0}}$ if and only if $f \in  \operatorname{B}^{1-\frac{\theta_{2}}{p}}_{p}(\mathcal{H}_{\theta_{2}}\lfloor_{S^{2}})$, $f^{\sharp}_{\mu\lfloor_{S^{1}}} \in L_{p}(S,\mu)$ and $\mathcal{GL}^{(3)}_{p}(f) < +\infty$.
Furthermore,
\begin{equation}
\|f|W_{p}^{1}(\operatorname{X})|_{S}^{\mathfrak{m}_{0}}\| \approx \|f|L_{p}(S^{1},\mu)\|+\|f^{\sharp}_{\mu\lfloor_{S^{1}}}|L_{p}(S^{1},\mu)\|+\|f|\operatorname{B}_{p}^{1-\frac{\theta_{2}}{p}}(S^{2})\|+\mathcal{GL}^{(3)}_{p}(f),
\end{equation}
where the equivalence constants do not depend on $f$.

Finally, there exists an $\mathfrak{m}_{0}$-extension operator $\operatorname{Ext}_{S,\{\mathfrak{m}_{k}\}} \in \mathcal{L}(W_{p}^{1}(\operatorname{X})|_{S}^{\mathfrak{m}_{0}},W_{p}^{1}(\operatorname{X}))$.
\end{Ca}

\begin{proof}
Applying Theorem \ref{Th.main} firstly for the set $S^{1}$ with the sequence of measures
$\{\mathfrak{m}^{1}_{k}\}:=\{2^{k\theta}\mu\}$, and secondly for the set $S^{2}$ with the sequence of measures $\{\mathfrak{m}^{2}_{k}\}:=\{2^{k(\theta-\theta_{2})}\mathcal{H}_{\theta_{2}}\lfloor_{S^{2}}\}$, we obtain the existence of a constant $C > 0$ such that
\begin{equation}
\notag
\|f|L_{p}(S^{1},\mu)\|+\|f^{\sharp}_{\mu\lfloor_{S^{1}}}|L_{p}(S^{1},\mu)\| + \|f|\operatorname{B}_{p}^{1-\frac{\theta_{2}}{p}}(S^{2})\| \le C\|f|W_{p}^{1}(\operatorname{X})|_{S}^{\mathfrak{m}_{0}}\|.
\end{equation}
Furthermore, it is clear that $\chi_{S^{2}}F|_{S}^{\mathfrak{m}_{0}}=F|_{S^{2}}^{\mathcal{H}_{\theta_{2}}}$ and $\chi_{S^{1} \setminus S^{2}}F|_{S}^{\mathfrak{m}_{0}}=\chi_{S^{1} \setminus S^{2}}F|_{S^{1}}^{\mu}$.
The proof concludes by combining the above observations with Theorems \ref{Th.3}, \ref{Th.4}.
\end{proof}

\subsection{Concluding remarks} It is worth to note that, in particular, our results are applicable to the situation when $S=\cup_{i=1}^{N}S^{i}$
and $S^{N} \subset ... \subset S^{1}=S$. Informally speaking, in this case we characterize the trace space of the $W_{p}^{1}(\operatorname{X})$ to the set $S^{1}$
with ``different accuracy'' on different pieces of $S$. Surprisingly, such cases have never been considered in the literature.

Note also that the concrete construction of the weights given in \eqref{eqq.weight} is irrelevant. Indeed, one can use other weights $\widetilde{\operatorname{w}}_{k}$ in all main 
results of the present paper. 
The only requirement is that, given $c \geq 1$, there exists a constant $C > 0$ such that the following inequalities
$$
C^{-1}\widetilde{\operatorname{w}}_{k}(y,z) \le \operatorname{w}_{k}(y,z) \le C \widetilde{\operatorname{w}}_{k}(y,z)
$$ 
hold for each $k \in \mathbb{N}_{0}$ for all $(y,z) \in \operatorname{X} \times \operatorname{X}$ satisfying $\operatorname{d}(y,z) \le c 2^{-k}$. For example, for each $k \in \mathbb{N}_{0}$, 
one can take
$$
\widetilde{\operatorname{w}}_{k}(y,z):=\frac{1}{2}\Bigl(\frac{1}{\mu(B_{k}(y))}+\frac{1}{\mu(B_{k}(z))}\Bigr), \quad (y,z) \in \operatorname{X} \times \operatorname{X}.
$$


%
%
\end{document}